\documentclass[12pt,a4paper]{article}
\usepackage{amsmath}
\usepackage{amsfonts}
\usepackage{amssymb}
\usepackage{amsthm}
\usepackage{mathtools}
\usepackage{mathabx}
\usepackage{xcolor}
\usepackage{hyperref}
\usepackage[latin1]{inputenc} 
\usepackage[french,english]{babel}  
\usepackage{paralist}
\usepackage[numbers]{natbib}
\usepackage{bbm}
\usepackage{pifont}
\usepackage{titlesec}

\titleformat{\paragraph}
{\normalfont\normalsize\bfseries}{\theparagraph}{1em}{}
\titlespacing*{\paragraph}
{0pt}{3.25ex plus 1ex minus .2ex}{1.5ex plus .2ex}

\newtheorem{proposition}{Proposition}
\newtheorem{definition}{Definition}
\newtheorem{theorem}{Theorem}
\newtheorem{remark}{Remark}

\newtheorem{assumption}{Assumption}
\newtheorem{property}{Property}

\newcommand{\N}{\mathbb{N}}

\newcommand{\cf}{\mathcal{F}}

\newcommand{\F}{\mathbb{F}}

\def\esssup{\text{ess sup}} 
\DeclareMathOperator*{\essinf}{ess\,inf}

\makeatletter
\renewenvironment{proof}[1][\proofname]{%
  \par\pushQED{\qed}\normalfont%
  \topsep6\p@\@plus6\p@\relax
  \trivlist\item[\hskip\labelsep\bfseries#1\@addpunct{.}]%
  \ignorespaces
}{%
  \popQED\endtrivlist\@endpefalse
}

\author{ 
 Miryana Grigorova\footnote{Corresponding Author. Department of Statistics, University of Warwick, E-mail:miryana.grigorova@warwick.ac.uk} \\
 \and
 Marie-Claire Quenez \footnote{LPSM, University Paris-Cit\'e} 
 \and
 Peng Yuan \footnote{Department of Statistics, University of Warwick} 
}

\begin{document}
\title{The non-linear multiple stopping problem: between the discrete and the continuous time}
\maketitle

\textit{Abstract:}
We consider the non-linear optimal  multiple stopping problem under general conditions on the non-linear evaluation operators, which might depend on two time indices: the time of evaluation/assessment and the horizon (when the reward or loss is incurred). We do not assume convexity/concavity or cash-invariance. We focus on the case where the agent's stopping strategies are what we call Bermudan stopping strategies, a framework which can be seen as lying between the discrete and the continuous time. We first study the non-linear double optimal stopping problem by using a reduction approach. We provide a necessary and a sufficient condition for optimal pairs, and a result on existence of optimal pairs. We then generalize the results to the non-linear $d$-optimal stopping problem. We treat the symmetric case (of additive and multiplicative reward families) as examples.

\section{Introduction}
The \textbf{linear  multiple}  stopping problem where the assessment functionals (resp. operators) are the usual linear expectations (resp. linear conditional expectations) has been studied in \cite{Kobylanski-Quenez-Mironescu-1} and \cite{Kobylanski-Quenez-Mironescu-2}, and applications, in particular to mathematical finance, have been provided in \cite{Latifa-Bonnans-Mnif}, \cite{Bender-Schoenmakers-Zhang}, \cite{Carmona-Touzi-1} together with a dual representation. Recently,  optimal multiple stopping problems have attracted renewed interest due to connections with mean-field optimal stopping (\cite{Talbi-Touzi-Zhang-1} and \cite{Talbi-Touzi-Zhang-2}) and to advances in numerical methods (\cite{Han-Li} and \cite{Magnino-Zhu-Lauriere}). On the other hand, \textbf{non-linear} optimal stopping problems and their applications have also been popular in the recent years (cf. e.g. \cite{Bayraktar-2}, \cite{Bayraktar}, \cite{Bayraktar-3}, \cite{Ekren}, \cite{Grigorova-3}, \cite{Grigorova-4}, \cite{Grigorova-Quenez-Peng-1}, \cite{Grigorova_Am}, \cite{Rutkowski}, \cite{Schoe}, \cite{Nutz}, \cite{Quenez-Kobylanski}, \cite{Quenez-Sulem}). There are fewer works on \textbf{non-linear multiple} stopping problems. In \cite{Li-1}, the author considers a non-linear optimal multiple  stopping problem with $g$-expectations, generated by a particular case of BSDEs, with driver $g$, which does not depend on $y$ and which is concave in $z$, and \cite{Li-2} studies the problem under the so-called $\mathbbm{F}$-expectation operators introduced in \cite{Bayraktar}, \cite{Bayraktar-3}. In both cases the non-linear operators depend on one index only (that is, the time of evaluation) and do not depend on the second index. Moreover, some additional properties on the operators hold true, such as: zero-one-law, translation invariance, local property, sub-additivity and positive homogeneity in the case of \cite{Li-2}, and those induced by the particular assumptions on the driver $g$ in the case of \cite{Li-1}. In the current paper, we work non-linear operators with two time indices under general assumptions (cf. also the works \cite{Grigorova-3}, \cite{Grigorova-4}, \cite{Grigorova-Quenez-Peng-1}, \cite{Grigorova-2} and \cite{Grigorova-Quenez-Peng-2} where some non-linear operators with two time indices also appear). Recently, some authors (cf., e.g., \cite{Nunno-Gianin}) have also emphasised the importance of ``horizon risk'' in mathematical finance, captured by operators depending also on the second time index which is the horizon. We place ourselves in a framework, where the agent's strategies are ``in-between'' the discrete stopping strategies and the continuous time stopping strategies.\\
[0.2cm]
The structure of the present paper is as follows: In Section \ref{Section2}, we introduce the framework, including the Bermudan stopping strategies and the non-linear operators $\rho_{S, \tau}[\cdot]$. In Section \ref{Section3}, we focus on the \textbf{non-linear double} stopping problem. We establish some basic properties of the value family. In Subsection \ref{SubSection3-1}, we study the problem via the so-called reduction approach (introduced in \cite{Kobylanski-Quenez-Mironescu-1} and \cite{Kobylanski-Quenez-Mironescu-2} in the case of the usual linear expectations). We also provide a sufficient condition and a necessary condition for optimality, and a result on the existence of an optimal pair $(\tau_{1}^{*}, \tau_{2}^{*})$. In Section \ref{Section4}, we present the \textbf{non-linear $d$-stopping} problem (where $d \geq 2$) and establish some basic properties of the value family. In Subsection \ref{SubSection4-1}, we study the problem via a reduction approach. We also present the particular case of a symmetric reward family (including the additive case and the multiplicative case). In Subsection \ref{SubSection4-2}, we establish a sufficient condition and a necessary condition for optimality for the general \textbf{non-linear $d$-stopping} problem. In Section \ref{Section5}, we provide examples of non-linear operators entering our framework. 

\section{The framework} \label{Section2}
Let $T>0$ be a \textbf{fixed finite} terminal horizon.\\
[0.2cm]
Let $(\Omega,  \cf, P)$ be a (complete) probability space equipped with a right-continuous complete filtration $\F = \{\mathcal{F}_t \colon t\in[0,T]\}$. \\
[0.2cm]
In the sequel, equalities and inequalities between random variables are to be understood in the $P$-almost sure sense. Equalities between measurable sets are to be understood in the $P$-almost sure sense.\\
[0.2cm]
Let $\N$ be the set of natural numbers, including $0$.  Let $\N^*$ be the set of natural numbers, excluding $0$.\\
[0.2cm] 
We first define what we call the Bermudan stopping strategies (which we used also in \cite{Grigorova-Quenez-Peng-1} for the case of the single agent optimal stopping problem).\\
[0.2cm]
Let ($\theta_k)_{k\in\N}$ be a sequence of stopping times satisfying the following properties:
\begin{itemize}
\item[(a)] The sequence $(\theta_k)_{k\in\N}$  is non-decreasing, i.e. 
for all $k\in\N$, $\theta_k\leq \theta_{k+1}$, a.s. 
\item[(b)] $\lim_{k\to\infty}\uparrow \theta_k=T$ a.s. 
\end{itemize}
Moreover, we set $\theta_0=0$. \\
[0.2cm]
We note that the family of $\sigma$-algebras 
$({\cal F}_{\theta_k})_{k\in\N}$ is non- decreasing (as the sequence $(\theta_{k})$ is non-decreasing).
We denote by ${\Theta}$ the set of stopping times $\tau$ of the form 
\begin{equation}\label{form}
\tau= \sum_{k=0}^{+\infty} \theta_k {\bf 1}_{A_k}+T{\bf 1}_{\bar A},
\end{equation}
where $\{(A_{k})^{+\infty}_{k=0}, \bar{A}\}$ form a partition of $\Omega$ such that, for each $k\in\N$,  $A_{k} \in \mathcal{F}_{\theta_{k}}$, and $\bar{A} \in \mathcal{F}_{T}$.\\
[0.2cm]
The set ${\Theta}$ can also  be described as the set of stopping times $\tau$ such that for almost all $\omega \in \Omega$, either $\tau(\omega) = T$ or $\tau(\omega) = \theta_{k}(\omega)$, for some $k = k(\omega) \in \mathbbm{N}$.\\
[0.2cm]
Note that the set $\Theta$ is closed under concatenation, that is, for each $\tau$ $\in$ $\Theta$ and each 
$A \in {\cal F}_{\tau}$, the stopping time $\tau {\bf 1}_{A} + T {\bf 1}_{A^c}$ $\in$ $\Theta$. More generally, for each $\tau$ $\in$ $\Theta$, $\tau'$ $\in$ $\Theta$ and each 
$A \in {\cal F}_{\tau\wedge \tau'}$, the stopping time $\tau {\bf 1}_{A} + \tau' {\bf 1}_{A^c}$ is in  $\Theta$.  The set $\Theta$ is also closed under pairwise minimization (that is, for  each $\tau\in\Theta$ and $\tau'\in\Theta$, we have $\tau\wedge \tau'\in\Theta$) and under pairwise maximization (that is, for  each $\tau\in\Theta$ and $\tau'\in\Theta$, we have $\tau\vee \tau'\in\Theta$). Moreover, the set $\Theta$ is closed under monotone limit, that is, for each non-decreasing (resp. non-increasing) sequence of stopping times $(\tau_n)_{n \in {\mathbb N}} \in \Theta^{\mathbbm{N}}$, we have $\lim_{n \rightarrow + \infty} \tau_n$  $\in$ $\Theta$.\\
[0.2cm]
We note also that all stopping times in $\Theta$ are bounded from above by $T$. 
\begin{remark}\label{Rk_canonical} 
We have the following \emph{canonical} writing of the sets in \eqref{form}: 
\begin{align*}
A_0&=\{\tau=\theta_0\};\\
A_{n+1}&=\{\tau=\theta_{n+1},\theta_{n+1}<T\}\backslash (A_n\cup...\cup A_0); \text { for all } n\in\N^*\\
\bar A&=(\cup_{k=0}^{+\infty} A_k)^c  
\end{align*}  
From this writing, we have: if $\omega\in A_{k+1}\cap \{\theta_k<T\}$, then $\omega\notin \{\tau=\theta_k\}.$ 
\end{remark}
\noindent
For each $\tau \in \Theta$, we denote by $\Theta_{\tau}$ the set of stopping times $\nu \in 
\Theta$ such that $\nu \geq \tau$ a.s.\, The set $\Theta_{\tau}$ satisfies the same properties as 
the set $\Theta$.
We will refer to the set $\Theta$ as the set of \textbf{Bermudan stopping strategies},
and to the set $\Theta_\tau$ as the set of Bermudan stopping strategies, greater than or equal to $\tau$ (or the set of Bermudan stopping strategies from time $\tau$ perspective). For simplicity, the set $\Theta_{\theta_{k}}$ will be denoted by $\Theta_{k}$.\\
[0.3cm]
\noindent
We recall the definition of an admissible family and of a bi-admissible family.

\begin{definition} \label{admissible}
We say that a family 
$\phi=(\phi(\tau), \, \tau \in \Theta)$  is \emph{admissible} if it satisfies the following conditions 
\par
1. \quad for all
$\tau \in \Theta$, $\phi(\tau)$ is a real valued random variable, which is  $\mathcal{F}_\tau$-measurable. \par
 2. \quad  for all
$\tau,\tau'\in \Theta$, $\phi(\tau)=\phi(\tau')$ a.s.  on
$\{\tau=\tau'\}$.\\
[0.2cm]
\noindent
Moreover, for $p \in [1, +\infty]$ fixed, we say that an admissible family $\phi$ is $p$-integrable, if for all $\tau \in \Theta$, $\phi(\tau)$ is in $L^{p}$.

\end{definition}

\begin{definition} \label{bi-admissible}
Let $(\psi(\tau_{1}, \tau_{2}))_{(\tau_{1}, \tau_{2}) \in \Theta \times \Theta}$ be a family of random variables doubly indexed by two Bermudan stopping times. We say that the family $\psi$ is \textbf{bi-admissible} if it satisfies the following properties:
\begin{itemize}
\item[(a)] For each $(\tau_{1}, \tau_{2}) \in \Theta \times \Theta$, the random variable $\psi(\tau_{1}, \tau_{2})$ is $\mathcal{F}_{\tau_{1}\vee\tau_{2}}$ - measurable.
\item[(b)] $\psi(\tau_{1}, \tau_{2}) = \psi(\tilde{\tau}_{1}, \tilde{\tau}_{2})$ on the set $\{\tau_{1} = \tilde{\tau}_{1}\} \cap \{\tau_{2} = \tilde{\tau}_{2}\}$.
\end{itemize}
Moreover, for $p \in [1, +\infty]$ fixed, we say that a bi-admissible family $\psi$ is $p$-integrable, if for all $\tau_{1} \in \Theta$, for all $\tau_{2} \in \Theta$, $\psi(\tau_{1}, \tau_{2}) \in L^{p}$.

\end{definition}
\noindent
Let $p \in [1, +\infty]$. We introduce the following properties  on the non-linear operators $\rho_{S, \tau}[\cdot]$, which will appear in the sequel. \\
[0.2cm] 
For
$S\in\Theta$, $S'\in\Theta$, $\tau\in\Theta$, for $\eta$, $\eta_{1}$ and $\eta_2$ in $L^p(\cf_\tau)$, for $\xi=(\xi(\tau))$ an admissible p-integrable family:  
\begin{compactenum}[(i)]
\item[(i)] $\rho_{S,\tau}: L^p(\cf_\tau)\longrightarrow L^p(\cf_S)$ 
\item[(ii)] \emph{(admissibility)} $\rho_{S,\tau}[\eta]=\rho_{S',\tau}[\eta]$ a.s.  on $\{S=S'\}$. 
\item[(iii)] \emph{(knowledge preservation)}
$\rho_{\tau,S}[\eta]=\eta, $
for all $\eta\in L^p(\cf_S)$, all $\tau\in\Theta_S.$
\item[(iv)] \emph{(monotonicity)}  $\;\rho_{S,\tau}[\eta_1]\leq \rho_{S,\tau}[\eta_2]$ a.s., if  $\eta_1\leq \eta_2$ a.s. 
\item[(v)] \emph{(consistency)}  $\;\rho_{S,\theta}[\rho_{\theta,\tau}[\eta]]= \rho_{S,\tau}[\eta]$, for all $S, \theta, \tau$ in $\Theta$ such that $S\leq \theta\leq \tau$ a.s.  
\item [(vi)] \emph{("generalized zero-one law") }$\; I_A\rho_{S,\tau}[\xi(\tau)] =I_A\rho_{S,\tau'}[\xi(\tau')],$ for all $A\in\cf_S$, $\tau\in \Theta_{S}$, $\tau'\in\Theta_{S}$ such that $\tau=\tau'$ on $A$.  
\item[(vi bis)] $\mathbbm{1}_{\{\tau_{1} = \tau_{1}'\}}\rho_{\tau_{1}, \tau_{1} \vee \tau_{2}}[\psi(\tau_{1}, \tau_{2})] = \mathbbm{1}_{\{\tau_{1} = \tau_{1}'\}}\rho_{\tau_{1}', \tau_{1}' \vee \tau_{2}}[\psi(\tau_{1}', \tau_{2})]$;\\[0.2cm]
$\mathbbm{1}_{\{\tau_{2} = \tau_{2}'\}}\rho_{\tau_{2}, \tau_{1} \vee \tau_{2}}[\psi(\tau_{1}, \tau_{2})] = \mathbbm{1}_{\{\tau_{2} = \tau_{2}'\}}\rho_{\tau_{2}', \tau_{1} \vee \tau_{2}'}[\psi(\tau_{1}, \tau_{2}')]$.
\item[(vi ter)] $\mathbbm{1}_{A}\rho_{\tau_{1} \wedge \tau_{2}, \tau_{1} \vee \tau_{2}}[\psi(\tau_{1} \wedge \tau_{2}, \tau_{1} \vee \tau_{2})] = \mathbbm{1}_{A}\rho_{\tau_{1} \wedge \tau_{2}, \tau_{1} \vee \tau_{2}}[\psi(\tau_{1}, \tau_{2})] = \\ 
\mathbbm{1}_{A}\rho_{\tau_{1}, \tau_{2}}[\psi(\tau_{1}, \tau_{2})],$ if $A$ is $\mathcal{F}_{\tau_{1} \wedge \tau_{2}}$-measurable, and $\tau_{1} \wedge \tau_{2} = \tau_{1}$ a.s. on $A$, and $\tau_{1} \vee \tau_{2} = \tau_{2}$ a.s. on $A$. Moreover, $\mathbbm{1}_{A}\rho_{\tau_{1} \wedge \tau_{2}, \tau_{1} \vee \tau_{2}}[\psi(\tau_{1} \vee \tau_{2}, \tau_{1} \wedge \tau_{2})] = \mathbbm{1}_{A}\rho_{\tau_{1} \wedge \tau_{2}, \tau_{1} \vee \tau_{2}}[\psi(\tau_{2}, \tau_{1})] = \mathbbm{1}_{A}\rho_{\tau_{1}, \tau_{2}}[\psi(\tau_{2}, \tau_{1})],$ if $A$ is $\mathcal{F}_{\tau_{1} \wedge \tau_{2}}$-measurable, and $\tau_{1} \wedge \tau_{2} = \tau_{1}$ a.s. on $A$, and $\tau_{1} \vee \tau_{2} = \tau_{2}$ a.s. on $A$.
\item[(vii)] \emph{(monotone Fatou property with respect to the terminal condition)}\\
$\rho_{S, \tau}[\eta] \leq \liminf_{n \to +\infty} \rho_{S, \tau}[\eta_{n}]$, for $(\eta_{n}), \eta$ such that $(\eta_{n})$ is non-decreasing, $\eta_{n} \in L^{p}(\cf_{\tau})$, $\sup_{n}\eta_{n} \in L^{p}$, and $\lim_{n \to +\infty} \uparrow \eta_{n} = \eta$ a.s.\\
\item[(viii)] (left-upper-semicontinuity (LUSC) along Bermudan stopping times with respect to the terminal condition and the terminal time), that is,  
$$\limsup_{n \to +\infty}\rho_{S, \tau_{n}}[\phi(\tau_{n})] \leq \rho_{S, \nu}[\limsup_{n \to +\infty} \phi(\tau_{n})],$$
$\text{for each non-decreasing sequence }(\tau_{n}) \in \Theta_{S}^{\N} \text{ such that }\lim_{n \to +\infty} \uparrow \tau_{n} = \nu \text{ a.s.},$ and for each $p$-integrable admissible family $\phi$ such that $\sup_{n\in\N} |\phi(\tau_n)|\in L^p.$\\
\item[(ix)] $\limsup_{n \to +\infty}\rho_{\theta_{n}, T}[\eta] \leq \rho_{T, T}[\eta], \text { for all } \eta\in L^p(\cf_T).$\\
\end{compactenum}

\noindent
Let us note that, if $\psi$ is symmetric, that is, $\psi(\tau_{1}, \tau_{2}) = \psi(\tau_{2}, \tau_{1})$, then the two properties in (vi bis) reduce to one equality.\\
[0.2cm]
In Section \ref{Section5}, we provide some examples of operators $\rho_{S, \tau}$ entering our framework.

\begin{remark} \label{TEXREmark2}
We will show that $\rho$ satisfies the following property
$$\mathbbm{1}_{A}\rho_{S, \tau_{1} \vee \tau_{2}}[\psi(\tau_{1}, \tau_{2})] = \mathbbm{1}_{A}\rho_{S, \tau_{1}' \vee \tau_{2}'}[\psi(\tau_{1}', \tau_{2}')]$$
for all $A \in \mathcal{F}_{S}, (\tau_{1}, \tau_{2}) \in \Theta_{S} \times \Theta_{S}$ and $(\tau_{1}', \tau_{2}') \in \Theta_{S} \times \Theta_{S}$ such that $\tau_{1} = \tau_{1}'$ on $A$, and $\tau_{2} = \tau_{2}'$ on $A$.\\
[0.2cm]
Indeed, by the consistency,  
$$\mathbbm{1}_{A}\rho_{S, \tau_{1} \vee \tau_{2}}[\psi(\tau_{1}, \tau_{2})] = \mathbbm{1}_{A}\rho_{S, \tau_{1}}[\rho _{\tau_{1}, \tau_{1} \vee \tau_{2}}[\psi(\tau_{1}, \tau_{2})]].$$
We note that, for $\tau_{2} \in \Theta$ fixed, the family $(\tilde{\phi}(\tau_{1}))_{\tau_{1} \in \Theta_{S}}$ is admissible (by property (vi bis)) and $p$-integrable, where $\tilde{\phi}(\tau_{1})$ is defined by:
$$\tilde{\phi}(\tau_{1}) \coloneqq \rho_{\tau_{1}, \tau_{1} \vee \tau_{2}}[\psi(\tau_{1}, \tau_{2})].$$
By applying the ``generalized zero-one law'', we get 
\begin{align*}
\mathbbm{1}_{A}\rho_{S, \tau_{1}}[\rho _{\tau_{1}, \tau_{1} \vee \tau_{2}}[\psi(\tau_{1}, \tau_{2})]] &= \mathbbm{1}_{A}\rho_{S, \tau_{1}'}[\rho _{\tau_{1}', \tau_{1}' \vee \tau_{2}}[\psi(\tau_{1}', \tau_{2})]]\\
&= \mathbbm{1}_{A}\rho_{S, \tau_{1}' \vee \tau_{2}}[\psi(\tau_{1}', \tau_{2})],
\end{align*}
where we have used the consistency property of $\rho$ for the last equality.\\
[0.2cm]
By the consistency property of $\rho$, 
$$\mathbbm{1}_{A}\rho_{S, \tau_{1}'\vee\tau_{2}}[\psi(\tau_{1}', \tau_{2})] = \mathbbm{1}_{A}\rho_{S, \tau_{2}}[\rho_{\tau_{2}, \tau_{1}' \vee \tau_{2}}[\psi(\tau_{1}', \tau_{2})]].$$
For $\tau_{1}'$ given, the family $(\tilde{\psi}(\tau_{2}))_{\tau_{2} \in \Theta_{S}}$, defined by, $\tilde{\psi}(\tau_{2}) \coloneqq \rho_{\tau_{2}, \tau_{1}' \vee \tau_{2}}[\psi(\tau_{1}', \tau_{2})]$, it is admissible (by property (vi bis) of $\rho$) and $p$-integrable.\\
[0.2cm]
Hence, by the ``generalized zero-one law'', we get 
$$\mathbbm{1}_{A}\rho_{S, \tau_{2}}[\rho_{\tau_{2}, \tau_{1}' \vee \tau_{2}}[\psi(\tau_{1}', \tau_{2})]] = \mathbbm{1}_{A}\rho_{S, \tau_{2}'}[\rho_{\tau_{2}', \tau_{1}' \vee \tau_{2}'}[\psi(\tau_{1}', \tau_{2}')]] = \mathbbm{1}_{A}\rho_{S, \tau_{1}' \vee \tau_{2}'}[\psi(\tau_{1}', \tau_{2}')],$$
where we have used the consistency property of $\rho$ for the last equality.\\
[0.2cm]
Finally, we get:
$$\mathbbm{1}_{A}\rho_{S, \tau_{1}\vee\tau_{2}}[\psi(\tau_{1}, \tau_{2})] = \mathbbm{1}_{A}\rho_{S, \tau_{1}'\vee\tau_{2}'}[\psi(\tau_{1}', \tau_{2}')].$$

\end{remark}

\begin{definition}\label{def_supermartingaleBIS}
Let $\phi=(\phi(\tau), \, \tau \in \Theta)$ be a \emph{$p$-integrable admissible} family.
 We say that $\phi$
 is a \emph{$(\Theta, \rho)$-supermartingale} (resp. \emph{$(\Theta, \rho)$--martingale}) \emph{ family} if 
 for all $\sigma, \tau$ in $\Theta$ such that $\sigma\leq\tau$ a.s., we have 
$$\rho_{\sigma,\tau}[\phi({\tau})]\leq \phi({\sigma}) \text{ (resp. }=\phi(\sigma)) \text{ a.s.} $$

\end{definition}

\section{The optimal non-linear double stopping problem} \label{Section3}
Let $(\psi(\tau_{1}, \tau_{2}))_{(\tau_{1}, \tau_{2}) \in \Theta \times \Theta}$ be a bi-admissible, $p$-integrable family.
We present the optimization problem of interest:\\
[0.2cm]
Let $S \in \Theta$ be given. We define $V(S)$ by
\begin{equation}\label{TEXEquation-3-222}
V(S) = \esssup_{(\tau_{1}, \tau_{2}) \in \Theta_{S} \times \Theta_{S}}\rho_{S, \tau_{1}\vee\tau_{2}}[\psi(\tau_{1}, \tau_{2})].
\end{equation}

\begin{assumption} \label{TEXAssumption3-111}
For each $S \in \Theta$, $V(S) \in L^{p}$.

\end{assumption}

\begin{proposition}\label{TEXproposition2}
i) $(V(S))_{S \in \Theta}$ is an admissible family.\\
[0.2cm]
ii) (maximising sequence property) There exists a sequence of pairs of Bermudan stopping times $(\tau_{1}^{n}, \tau_{2}^{n}) \in \Theta_{S} \times \Theta_{S}$ such that 
\begin{itemize}
\item{} $(\rho_{S, \tau_{1}^{n} \vee \tau_{2}^{n}}[\psi(\tau_{1}^{n}, \tau_{2}^{n})])$ is non-decreasing and,
\item{} $V(S) = \lim_{n \to +\infty} \uparrow \rho_{S, \tau_{1}^{n} \vee \tau_{2}^{n}}[\psi(\tau_{1}^{n}, \tau_{2}^{n})]$.
\end{itemize}
iii) The value family $(V(S))_{S \in \Theta}$ is a $(\Theta, \rho)$ - supermartingale family.

\end{proposition}

\begin{proof}
\textbf{(i)} For each $S \in \Theta$, $V(S)$ is $\mathcal{F}_{S}$-measurable; this is due to property (i) of $\rho$ and to a well-known property of the essential supremum.\\
[0.2cm]
Let $S, S' \in \Theta$. We set $A = \{S = S'\}$ and we will show that $V(S) = V(S')$ $P-a.s.$ on $A$.\\
[0.2cm]
We have 
\begin{align*}
\mathbbm{1}_{A}V(S) &= \mathbbm{1}_{A}\esssup_{(\tau_{1}, \tau_{2}) \in \Theta_{S} \times \Theta_{S}} \rho_{S, \tau_{1} \vee \tau_{2}}[\psi(\tau_{1}, \tau_{2})]\\
&= \esssup_{(\tau_{1}, \tau_{2}) \in \Theta_{S} \times \Theta_{S}} \mathbbm{1}_{A}\rho_{S, \tau_{1} \vee \tau_{2}}[\psi(\tau_{1}, \tau_{2})]\\
&= \esssup_{(\tau_{1}, \tau_{2}) \in \Theta_{S} \times \Theta_{S}} \mathbbm{1}_{A}\rho_{S', \tau_{1} \vee \tau_{2}}[\psi(\tau_{1}, \tau_{2})],
\end{align*}
where for the last equality we have used the admissibility property of $\rho$ (property (ii)).\\
[0.2cm]
Let $\tau_{1} \in \Theta_{S}$ and $\tau_{2} \in \Theta_{S}$. Let $\tau_{1}^{A} \coloneqq \tau_{1}\mathbbm{1}_{A} + T\mathbbm{1}_{A^{c}}$ and $\tau_{2}^{A} \coloneqq \tau_{2}\mathbbm{1}_{A} + T\mathbbm{1}_{A^{c}}$. Then, $\tau_{1}^{A}$ and $\tau_{2}^{A}$ are in $\Theta_{S'}$ (where we have used the property of stability by concatenation of $\Theta$). Moreover, $\tau_{1}^{A} = \tau_{1}$ on $A$ and $\tau_{2}^{A} = \tau_{2}$ on $A$. Hence, $\tau_{1}^{A} \vee \tau_{2}^{A} = \tau_{1} \vee \tau_{2}$ on $A$. Thus, by Remark \ref{TEXREmark2} we get
\begin{align*}
\mathbbm{1}_{A}\rho_{S', \tau_{1} \vee \tau_{2}}[\psi(\tau_{1}, \tau_{2})] &= \mathbbm{1}_{A}\rho_{S', \tau_{1}^{A} \vee \tau_{2}^{A}}[\psi(\tau_{1}^{A}, \tau_{2}^{A})]\\
&\leq \mathbbm{1}_{A}V(S').
\end{align*}
We have shown that $\mathbbm{1}_{A}V(S) \leq \mathbbm{1}_{A}V(S')$. We obtain the converse inequality by interchanging the roles of $S$ and $S'$. This ends the proof of (i).\\
[0.2cm]
\textbf{(ii)} It is sufficient to show that the family $(\rho_{S, \tau_{1} \vee \tau_{2}}[\psi(\tau_{1}, \tau_{2})])_{(\tau_{1}, \tau_{2}) \in \Theta_{S} \times \Theta_{S}}$ is directed upwards.\\
[0.2cm]
Let $(\tau_{1}, \tau_{2}) \in \Theta_{S} \times \Theta_{S}$ and $(\tau_{1}', \tau_{2}') \in \Theta_{S} \times \Theta_{S}$. We define the set $A \coloneqq \{\rho_{S, \tau_{1}' \vee \tau_{2}'}[\psi(\tau_{1}', \tau_{2}')] \leq \rho_{S, \tau_{1} \vee \tau_{2}}[\psi(\tau_{1}, \tau_{2})]\}$. Trivially, $A \in \mathcal{F}_{S}$.\\
[0.2cm] 
We set: $\nu_{1} = \tau_{1}\mathbbm{1}_{A} + \tau_{1}'\mathbbm{1}_{A^{c}}$ and $\nu_{2} = \tau_{2}\mathbbm{1}_{A} + \tau_{2}'\mathbbm{1}_{A^{c}}$. We have: $(\nu_{1}, \nu_{2}) \in \Theta_{S} \times \Theta_{S}$. \\
[0.2cm]
We have by Remark \ref{TEXREmark2}:
\begin{align*}
\rho_{S, \nu_{1} \vee \nu_{2}}[\psi(\nu_{1}, \nu_{2})] &= \rho_{S, \nu_{1} \vee \nu_{2}}[\psi(\nu_{1}, \nu_{2})]\times \mathbbm{1}_{A} + \rho_{S, \nu_{1} \vee \nu_{2}}[\psi(\nu_{1}, \nu_{2})] \times \mathbbm{1}_{A^{c}}\\
&= \rho_{S, \tau_{1} \vee \tau_{2}}[\psi(\tau_{1}, \tau_{2})]\times \mathbbm{1}_{A} + \rho_{S, \tau_{1}' \vee \tau_{2}'}[\psi(\tau_{1}', \tau_{2}')] \times \mathbbm{1}_{A^{c}}\\
&= \max(\rho_{S, \tau_{1} \vee \tau_{2}}[\psi(\tau_{1}, \tau_{2})], \rho_{S, \tau_{1}' \vee \tau_{2}'}[\psi(\tau_{1}', \tau_{2}')]).
\end{align*}
Therefore, the family $(\rho_{S, \tau_{1} \vee \tau_{2}}[\psi(\tau_{1}, \tau_{2})])_{(\tau_{1}, \tau_{2}) \in \Theta_{S} \times \Theta_{S}}$ is directed upwards. Hence, by a well-known property of the essential supremum, statement (ii) holds true.\\
[0.2cm]
\textbf{(iii)} The family $(V(S))_{S \in \Theta}$ is admissible (by statement (i)) and $p$-integrable (by Assumption \ref{TEXAssumption3-111}). Let $S, S' \in \Theta$ be such that $S \leq S'$ a.s. By statement (ii), there exists a sequence $(\tau_{1}^{n}, \tau_{2}^{n}) \in \Theta_{S'} \times \Theta_{S'}$ such that $V(S') = \lim_{n \to +\infty}\uparrow\rho_{S', \tau_{1}^{n} \vee \tau_{2}^{n}}[\psi(\tau_{1}^{n}, \tau_{2}^{n})]$. Hence, 
\begin{align*}
\rho_{S, S'}[V(S')] &= \rho_{S, S'}[\lim_{n \to +\infty}\uparrow\rho_{S', \tau_{1}^{n} \vee \tau_{2}^{n}}[\psi(\tau_{1}^{n}, \tau_{2}^{n})]]\\
&\leq \liminf_{n \to +\infty}\rho_{S, S'}[\rho_{S', \tau_{1}^{n} \vee \tau_{2}^{n}}[\psi(\tau_{1}^{n}, \tau_{2}^{n})]],
\end{align*}
where we have used the monotone Fatou property of $\rho_{S, S'}$ with respect to the terminal condition (property (vii)).\\
[0.2cm]
Now, by the consistency property, 
\begin{align*}
\rho_{S, S'}[\rho_{S', \tau_{1}^{n} \vee \tau_{2}^{n}}[\psi(\tau_{1}^{n}, \tau_{2}^{n})]] = \rho_{S, \tau_{1}^{n} \vee \tau_{2}^{n}}[\psi(\tau_{1}^{n}, \tau_{2}^{n})].
\end{align*}
Finally, 
\begin{align*}
\rho_{S, S'}[V(S')] &\leq  \liminf_{n \to +\infty}\rho_{S, \tau_{1}^{n} \vee \tau_{2}^{n}}[\psi(\tau_{1}^{n}, \tau_{2}^{n})] \leq V(S).
\end{align*}
Hence, $(V(S))_{S \in \Theta}$ is a $(\Theta, \rho)$-supermartingale family.

\end{proof}

\subsection{Reduction approach} \label{SubSection3-1}
We define the following two auxiliary problems: for each $S \in \Theta$,
\begin{equation} \label{TEXEQuation6666}
v_{1}(S) \coloneqq \esssup_{\tau_{1} \in \Theta_{S}} \rho_{S, \tau_{1}}[\psi(\tau_{1}, S)],
\end{equation}
and 
\begin{equation}
v_{2}(S) \coloneqq \esssup_{\tau_{2} \in \Theta_{S}} \rho_{S, \tau_{2}}[\psi(S, \tau_{2})].
\end{equation}
Since $\psi$ is bi-admissible, $(\psi(\tau_{1}, S))_{\tau_{1} \in \Theta_{S}}$ is admissible and $(\psi(S, \tau_{2}))_{\tau_{2} \in \Theta_{S}}$ is admissible. Moreover, $(\psi(\tau_{1}, S))_{\tau_{1} \in \Theta_{S}}$ is $p$-integrable and $(\psi(S, \tau_{2}))_{\tau_{2} \in \Theta_{S}}$ is $p$-integrable (as $(\psi(\tau_{1}, \tau_{2}))_{(\tau_{1}, \tau_{2}) \in \Theta \times \Theta}$ is $p$-integrable). Furthermore, by definition of $v_{1}, v_{2}$ and $V$, $\psi(S, S) \leq v_{1}(S) \leq V(S)$ a.s. for each $S \in \Theta$ and $\psi(S, S) \leq v_{2}(S) \leq V(S)$ a.s. for each $S \in \Theta$. Hence, by Assumption \ref{TEXAssumption3-111}, $v_{1}(S) \in L^{p}$, for each $S \in \Theta$ and $v_{2}(S) \in L^{p}$, for each $S \in \Theta$.\\
[0.2cm]
Hence, by the results of \cite{Grigorova-Quenez-Peng-1} on the single agent's non-linear optimal stopping problem, the value family $(v_{1}(S))_{S \in \Theta}$ and the value family $(v_{2}(S))_{S \in \Theta}$ are admissible under our assumptions on $\rho$, and we can apply the results of \cite{Grigorova-Quenez-Peng-1} to characterize these value families.\\
[0.2cm]
We define for each $\tau \in \Theta$, $\phi(\tau) \coloneqq v_{1}(\tau) \vee v_{2}(\tau)$. We now consider the optimal stopping problem with this auxiliary reward (or pay-off) family, that is, for each $S \in \Theta$, we consider
\begin{equation} \label{TEXEQuation8888}
u(S) \coloneqq \esssup_{\tau \in \Theta_{S}} \rho_{S, \tau}[\phi(\tau)].
\end{equation}

\begin{theorem}[Reduction] \label{TEXTheorem5-1}
For each $S \in \Theta$, $u(S) = V(S)$.

\end{theorem}

\begin{proof}
\textbf{Step 1:} We will show that $(V(S))_{S \in \Theta}$ is a $(\Theta, \rho)$-supermartingale greater than or equal to $(\phi(S))_{S \in \Theta}$.\\
[0.2cm]
We have already checked that $(V(S))_{S \in \Theta}$ is a $(\Theta, \rho)$-supermartingale (cf. Proposition \ref{TEXproposition2}). By definition of $V(S)$, for any $\tau_{1} \in \Theta_{S}$, 
$$V(S) \geq \rho_{S, \tau_{1} \vee S}[\psi(\tau_{1}, S)] = \rho_{S, \tau_{1}}[\psi(\tau_{1}, S)].$$
Hence, 
$$V(S) \geq \esssup_{\tau_{1} \in \Theta_{S}}\rho_{S, \tau_{1}}[\psi(\tau_{1}, S)] = v_{1}(S).$$
A similar argument leads to $V(S) \geq v_{2}(S)$.\\
[0.2cm]
Hence, $V(S) \geq v_{1}(S) \vee v_{2}(S) = \phi(S)$.\\
[0.2cm]
Hence, by the $(\Theta, \rho)$-Snell envelope characterization of the value family $(u(S))_{S \in \Theta}$ of problem \eqref{TEXEQuation8888}, we have $V(S) \geq u(S)$.\\
[0.2cm]
\textbf{Step 2:} We will show the converse inequality.\\
[0.2cm]
Let $\tau_{1}, \tau_{2} \in \Theta_{S}$. We will show that $\rho_{S, \tau_{1} \vee \tau_{2}}[\psi(\tau_{1}, \tau_{2})] \leq \rho_{S, \tau_{1} \wedge \tau_{2}}[\phi(\tau_{1} \wedge \tau_{2})]$.\\
[0.2cm]
Let $A \coloneqq \{\tau_{1} \leq \tau_{2}\}$. By the consistency property of $\rho$, we have: 
\begin{align*}
\rho_{S, \tau_{1} \vee \tau_{2}}[\psi(\tau_{1}, \tau_{2})] &= \rho_{S, \tau_{1} \wedge \tau_{2}}[\rho_{\tau_{1} \wedge \tau_{2}, \tau_{1} \vee \tau_{2}}[\psi(\tau_{1}, \tau_{2})]]\\
&= \rho_{S, \tau_{1} \wedge \tau_{2}}[\mathbbm{1}_{A}\rho_{\tau_{1} \wedge \tau_{2}, \tau_{1} \vee \tau_{2}}[\psi(\tau_{1}, \tau_{2})] + \mathbbm{1}_{A^{c}}\rho_{\tau_{1} \wedge \tau_{2}, \tau_{1} \vee \tau_{2}}[\psi(\tau_{1}, \tau_{2})]].
\end{align*}
Now, by property (vi ter), we have:
$$\mathbbm{1}_{A}\rho_{\tau_{1} \wedge \tau_{2}, \tau_{1} \vee \tau_{2}}[\psi(\tau_{1}, \tau_{2})] = \mathbbm{1}_{A}\rho_{\tau_{1}, \tau_{2}}[\psi(\tau_{1}, \tau_{2})] \leq \mathbbm{1}_{A}v_{2}(\tau_{1}),$$
and 
$$\mathbbm{1}_{A^{c}}\rho_{\tau_{1} \wedge \tau_{2}, \tau_{1} \vee \tau_{2}}[\psi(\tau_{1}, \tau_{2})] = \mathbbm{1}_{A^{c}}\rho_{\tau_{2}, \tau_{1}}[\psi(\tau_{1}, \tau_{2})] \leq \mathbbm{1}_{A^{c}}v_{1}(\tau_{2}).$$
Hence, by the monotonicity of $\rho_{S, \tau_{1} \wedge \tau_{2}}$, we get:
\begin{equation}\label{TEXEquation9999}
\begin{aligned}
\rho_{S, \tau_{1} \vee \tau_{2}}[\psi(\tau_{1}, \tau_{2})] &\leq \rho_{S, \tau_{1} \wedge \tau_{2}}[\mathbbm{1}_{A}v_{2}(\tau_{1}) + \mathbbm{1}_{A^{c}}v_{1}(\tau_{2})]\\
&\leq \rho_{S, \tau_{1} \wedge \tau_{2}}[\mathbbm{1}_{A}\phi(\tau_{1}) + \mathbbm{1}_{A^{c}}\phi(\tau_{2})]\\
&= \rho_{S, \tau_{1} \wedge \tau_{2}}[\mathbbm{1}_{A}\phi(\tau_{1} \wedge \tau_{2}) + \mathbbm{1}_{A^{c}}\phi(\tau_{2} \wedge \tau_{2})] = \rho_{S, \tau_{1} \wedge \tau_{2}}[\phi(\tau_{1} \wedge \tau_{2})] \leq u(S).
\end{aligned}
\end{equation}
We have thus obtained $\rho_{S, \tau_{1} \vee \tau_{2}}[\psi(\tau_{1}, \tau_{2})] \leq u(S)$.\\
[0.2cm]
Hence, 
$$V(S) = \esssup_{(\tau_{1}, \tau_{2}) \in \Theta_{S} \times \Theta_{S}}\rho_{S, \tau_{1} \vee \tau_{2}}[\psi(\tau_{1}, \tau_{2})] \leq u(S).$$
Form Step 1 and Step 2 we conclude $V(S) = u(S)$.

\end{proof}
\noindent
We now provide a sufficient condition for optimality.
\begin{proposition} [Construction of optimal stopping times for $V(S)$. Sufficient condition] \label{sufficient-condition}
We assume that:\\
[0.2cm]
i) $\theta^{*} \in \Theta_{S}$ is optimal for the problem with value $u(S)$,\\
[0.2cm]
ii) $\theta_{2}^{*} \in \Theta_{\theta^{*}}$ is optimal for the problem with value $v_{2}(\theta^{*})$,\\
[0.2cm]
iii) $\theta_{1}^{*} \in \Theta_{\theta^{*}}$ is optimal for the problem with value $v_{1}(\theta^{*})$.\\
[0.2cm]
Let $B \coloneqq \{v_{1}(\theta^{*}) \leq v_{2}(\theta^{*})\}$. Then, the pair $(\tau_{1}^{*}, \tau_{2}^{*})$ defined by:\\
[0.2cm]
$\tau_{1}^{*} \coloneqq \theta^{*}\mathbbm{1}_{B} + \theta_{1}^{*}\mathbbm{1}_{B^{c}}$, and $\tau_{2}^{*} \coloneqq \theta_{2}^{*}\mathbbm{1}_{B} + \theta^{*}\mathbbm{1}_{B^{c}}$ is optimal for $V(S)$.

\end{proposition}

\begin{proof}
We have, $\tau_{1}^{*} \wedge \tau_{2}^{*} = \theta^{*}$. On $B$, $\tau_{1}^{*} \wedge \tau_{2}^{*} = \theta^{*} = \tau_{1}^{*}$, and $\tau_{1}^{*} \vee \tau_{2}^{*} = \theta_{2}^{*} = \tau_{2}^{*}$. Moreover, on $B^{c}$, $\tau_{1}^{*} \wedge \tau_{2}^{*} = \theta^{*} = \tau_{2}^{*}$, and $\tau_{1}^{*} \vee \tau_{2}^{*} = \theta_{1}^{*} = \tau_{1}^{*}$. We have, by i), and by definition of $\phi$,
\begin{align*}
u(S) = \rho_{S, \theta^{*}}[\phi(\theta^{*})] &= \rho_{S, \tau_{1}^{*} \wedge \tau_{2}^{*}}[\phi(\tau_{1}^{*} \wedge \tau_{2}^{*})]\\
&= \rho_{S, \tau_{1}^{*} \wedge \tau_{2}^{*}}[v_{1}(\tau_{1}^{*} \wedge \tau_{2}^{*}) \vee v_{2}(\tau_{1}^{*} \wedge \tau_{2}^{*})].
\end{align*}
Hence,
\begin{align*}
u(S) &= \rho_{S, \tau_{1}^{*} \wedge \tau_{2}^{*}}[\mathbbm{1}_{B}v_{2}(\tau_{1}^{*} \wedge \tau_{2}^{*}) + \mathbbm{1}_{B^{c}}v_{1}(\tau_{1}^{*} \wedge \tau_{2}^{*})]\\
&= \rho_{S, \tau_{1}^{*} \wedge \tau_{2}^{*}}[\mathbbm{1}_{B}v_{2}(\tau_{1}^{*}) + \mathbbm{1}_{B^{c}}v_{1}(\tau_{2}^{*})]\\
&= \rho_{S, \tau_{1}^{*} \wedge \tau_{2}^{*}}[\mathbbm{1}_{B}v_{2}(\theta^{*}) + \mathbbm{1}_{B^{c}}v_{1}(\theta^{*})]\\
&= \rho_{S, \tau_{1}^{*} \wedge \tau_{2}^{*}}[\mathbbm{1}_{B}\rho_{\theta^{*}, \theta_{2}^{*}}[\psi(\theta^{*}, \theta_{2}^{*})] + \mathbbm{1}_{B^{c}}\rho_{\theta^{*}, \theta_{1}^{*}}[\psi(\theta_{1}^{*}, \theta^{*})]],
\end{align*}
where we have used the optimality of $\theta_{2}^{*}$ and $\theta_{1}^{*}$ from ii) and iii).\\
[0.2cm]
Now, by property (vi ter) of $\rho$ we have:
\begin{align*}
\mathbbm{1}_{B}\rho_{\theta^{*}, \theta_{2}^{*}}[\psi(\theta^{*}, \theta_{2}^{*})] = \mathbbm{1}_{B}\rho_{\tau_{1}^{*} \wedge \tau_{2}^{*}, \tau_{1}^{*} \vee \tau_{2}^{*}}[\psi(\tau_{1}^{*}, \tau_{2}^{*})],
\end{align*}
and 
\begin{align*}
\mathbbm{1}_{B^{c}}\rho_{\theta^{*}, \theta_{1}^{*}}[\psi(\theta_{1}^{*}, \theta^{*})] = \mathbbm{1}_{B^{c}}\rho_{\tau_{1}^{*} \wedge \tau_{2}^{*}, \tau_{1}^{*} \vee \tau_{2}^{*}}[\psi(\tau_{1}^{*}, \tau_{2}^{*})].
\end{align*}
Hence, we get 
\begin{align*}
u(S) &= \rho_{S, \tau_{1}^{*} \wedge \tau_{2}^{*}}[\mathbbm{1}_{B}\rho_{\tau_{1}^{*} \wedge \tau_{2}^{*}, \tau_{1}^{*} \vee \tau_{2}^{*}}[\psi(\tau_{1}^{*}, \tau_{2}^{*})] + \mathbbm{1}_{B^{c}}\rho_{\tau_{1}^{*} \wedge \tau_{2}^{*}, \tau_{1}^{*} \vee \tau_{2}^{*}}[\psi(\tau_{1}^{*}, \tau_{2}^{*})]]\\
&= \rho_{S, \tau_{1}^{*} \wedge \tau_{2}^{*}}[\rho_{\tau_{1}^{*} \wedge \tau_{2}^{*}, \tau_{1}^{*} \vee \tau_{2}^{*}}[\psi(\tau_{1}^{*}, \tau_{2}^{*})]]\\
&= \rho_{S, \tau_{1}^{*} \vee \tau_{2}^{*}}[\psi(\tau_{1}^{*}, \tau_{2}^{*})],
\end{align*}
where we have used the time consistency property of $\rho$ for the last equality.\\
[0.2cm]
From this together with Theorem \ref{TEXTheorem5-1}, we get
$$V(S) = u(S) = \rho_{S, \tau_{1}^{*} \vee \tau_{2}^{*}}[\psi(\tau_{1}^{*}, \tau_{2}^{*})],$$
which shows the optimality of the pair $(\tau_{1}^{*}, \tau_{2}^{*})$ for the problem with value $V(S)$.

\end{proof}

\begin{proposition}[A necessary condition for optimality]\label{double-necessary-condition} Let $S \in \Theta$. Suppose that a given pair $(\tau_{1}^{*}, \tau_{2}^{*})$ is optimal for $V(S)$. Let $A \coloneqq \{\tau_{1}^{*} \leq \tau_{2}^{*}\}$. Then, the following holds:\\
[0.2cm]
i) $\tau_{1}^{*} \wedge \tau_{2}^{*}$ is optimal for the problem with value $u(S)$.\\
[0.2cm]
Moreover, if $\rho$ is \textbf{strictly monotone}, then:\\
[0.2cm]
ii) $\tau_{2}^{*}$ is optimal for the problem with value $v_{2}(\tau_{1}^{*})$ a.s. on $A$.\\
[0.2cm]
iii) $\tau_{1}^{*}$ is optimal for the problem with value $v_{1}(\tau_{2}^{*})$ a.s. on $A^{c}$.

\end{proposition}

\begin{proof}
We have: 
\begin{equation}\label{TEXEquation1010}
V(S) = \rho_{S, \tau_{1}^{*} \vee \tau_{2}^{*}}[\psi(\tau_{1}^{*}, \tau_{2}^{*})] = u(S).
\end{equation}
By following the same arguments as in \textbf{Step 2} of the proof of Theorem \ref{TEXTheorem5-1}, with $\tau_{1}^{*}$ in place of $\tau_{1}$, and $\tau_{2}^{*}$ in place of $\tau_{2}$, we get:
\begin{align*}
\rho_{S, \tau_{1}^{*} \vee \tau_{2}^{*}}[\psi(\tau_{1}^{*}, \tau_{2}^{*})] &= \rho_{S, \tau_{1}^{*} \wedge \tau_{2}^{*}}[\mathbbm{1}_{A}\rho_{\tau_{1}^{*}, \tau_{2}^{*}}[\psi(\tau_{1}^{*}, \tau_{2}^{*})] + \mathbbm{1}_{A^{c}}\rho_{\tau_{2}^{*}, \tau_{1}^{*}}[\psi(\tau_{1}^{*}, \tau_{2}^{*})]]\\
&\leq \rho_{S, \tau_{1}^{*} \wedge \tau_{2}^{*}}[\mathbbm{1}_{A}v_{2}(\tau_{1}^{*}) + \mathbbm{1}_{A^{c}}v_{1}(\tau_{2}^{*})]\\
&\leq \rho_{S, \tau_{1}^{*} \wedge \tau_{2}^{*}}[\phi(\tau_{1}^{*} \wedge \tau_{2}^{*})] \leq u(S).
\end{align*}
Now, by Eq. \eqref{TEXEquation1010}, for the LHS, it holds: $\rho_{S, \tau_{1}^{*} \vee \tau_{2}^{*}}[\psi(\tau_{1}^{*}, \tau_{2}^{*})] = u(S)$. Hence, all the inequalities in the above computation are equalities.\\
[0.2cm]
Hence, $\tau_{1}^{*} \wedge \tau_{2}^{*}$ is optimal for $u(S)$. Moreover,
$$\rho_{S, \tau_{1}^{*} \wedge \tau_{2}^{*}}[\mathbbm{1}_{A}\rho_{\tau_{1}^{*}, \tau_{2}^{*}}[\psi(\tau_{1}^{*}, \tau_{2}^{*})] + \mathbbm{1}_{A^{c}}\rho_{\tau_{2}^{*}, \tau_{1}^{*}}[\psi(\tau_{1}^{*}, \tau_{2}^{*})]] = \rho_{S, \tau_{1}^{*} \wedge \tau_{2}^{*}}[\mathbbm{1}_{A}v_{2}(\tau_{1}^{*}) + \mathbbm{1}_{A^{c}}v_{1}(\tau_{2}^{*})].$$
Furthermore, by definition of $v_{2}(\tau_{1}^{*})$ and $v_{1}(\tau_{2}^{*})$, we have:
$$\mathbbm{1}_{A}\rho_{\tau_{1}^{*}, \tau_{2}^{*}}[\psi(\tau_{1}^{*}, \tau_{2}^{*})] + \mathbbm{1}_{A^{c}}\rho_{\tau_{2}^{*}, \tau_{1}^{*}}[\psi(\tau_{1}^{*}, \tau_{2}^{*})] \leq \mathbbm{1}_{A}v_{2}(\tau_{1}^{*}) + \mathbbm{1}_{A^{c}}v_{1}(\tau_{2}^{*}).$$
Hence, under the additional assumption that $\rho_{S, \tau_{1}^{*} \wedge \tau_{2}^{*}}$ is \textbf{strictly monotone}, we get
$$\mathbbm{1}_{A}\rho_{\tau_{1}^{*}, \tau_{2}^{*}}[\psi(\tau_{1}^{*}, \tau_{2}^{*})] + \mathbbm{1}_{A^{c}}\rho_{\tau_{2}^{*}, \tau_{1}^{*}}[\psi(\tau_{1}^{*}, \tau_{2}^{*})] = \mathbbm{1}_{A}v_{2}(\tau_{1}^{*}) + \mathbbm{1}_{A^{c}}v_{1}(\tau_{2}^{*}).$$
Hence, statements ii) and iii) hold true.

\end{proof}

\subsubsection{Existence of an optimal pair $(\tau_{1}^{*}, \tau_{2}^{*})$ for $V(S)$}
Let $S \in \Theta$. We define $\theta^{*}$ by $\theta^{*} \coloneqq \essinf\{\tau \geq S: u(\tau) = \phi(\tau) a.s.\}$, and $\theta_{1}^{*} \coloneqq \essinf\{\tau \geq \theta^{*}: v_{1}(\tau) = \psi(\tau, \theta^{*})\}$, $\theta_{2}^{*} \coloneqq \essinf\{\tau \geq \theta^{*}: v_{2}(\tau) = \psi(\theta^{*}, \tau)\}$.

\begin{theorem}[Existence of an optimal pair]\label{Existence-of-an-optimal-pair}
We assume that $\rho$ satisfies conditions (i) - (ix). Under Assumption \ref{TEXAssumption3-111}, the pair $(\tau_{1}^{*}, \tau_{2}^{*})$ defined by $\tau_{1}^{*} \coloneqq \theta^{*}\mathbbm{1}_{B} + \theta_{1}^{*}\mathbbm{1}_{B^{c}}$, and $\tau_{2}^{*} \coloneqq \theta_{2}^{*}\mathbbm{1}_{B} + \theta^{*}\mathbbm{1}_{B^{c}}$, where $B \coloneqq \{v_{1}(\theta^{*}) \leq v_{2}(\theta^{*})\}$, is an optimal pair for $V(S)$. 

\end{theorem}

\noindent
The idea of the proof is to use the result on existence of optimal stopping problems for the \textbf{single} non-linear optimal stopping problem (cf. Theorem 2.3 in \cite{Grigorova-Quenez-Peng-1}).\\
[0.2cm]
To simplify the presentation and to ensure that the value families of the relevant single optimal stopping problems are LUSC along Bermudan stopping times (cf. Assumption 2.2 in \cite{Grigorova-Quenez-Peng-1} and Proposition 2.3 in \cite{Grigorova-Quenez-Peng-1}), we will assume that the bi-admissible pay-off family $\psi$ satisfies: 

\begin{assumption}
i) $\limsup_{n \to +\infty}\phi(\theta_{n}) \leq \phi(T)$,\\
[0.2cm] 
ii) $\limsup_{n \to +\infty}\psi(\theta, \theta^{*}) \leq \psi(T, \theta^{*})$,\\
[0.2cm] 
iii) $\limsup_{n \to +\infty}\psi(\theta^{*}, \theta_{n}) \leq \psi(\theta^{*}, T)$.

\end{assumption}

\begin{remark}\label{remarkfortheorem2}
When ii) is satisfied, then, by Lemma 2.9 in \cite{Grigorova-Quenez-Peng-1}, we have:
$$\limsup_{n \to +\infty}v_{1}(\theta_{n}) \leq v_{1}(T).$$
When iii) is satisfied, then by Lemma 2.9 in \cite{Grigorova-Quenez-Peng-1}, we have:
$$\limsup_{n \to +\infty}v_{2}(\theta_{n}) \leq v_{2}(T).$$
When ii) and iii) are satisfied, we have:
\begin{align*}
\limsup_{n \to +\infty}\phi(\theta_{n}) &= \limsup_{n \to +\infty}(v_{1}(\theta_{n}) \vee v_{2}(\theta_{n}))\\ 
&\leq \limsup_{n \to +\infty}v_{1}(\theta_{n}) \vee \limsup_{n \to +\infty}v_{2}(\theta_{n}) \leq v_{1}(T) \vee v_{2}(T) = \phi(T).
\end{align*}
Hence, if ii) and iii) are satisfied, then, i) is also satisfied.

\end{remark}

\begin{proof}[Proof of Theorem \ref{Existence-of-an-optimal-pair}]
By Theorem 2.3 and Proposition 2.3 in \cite{Grigorova-Quenez-Peng-1}, the stopping time $\theta_{1}^{*}$ is optimal for the problem with value $v_{1}(\theta^{*})$, and $\theta_{2}^{*}$ is optimal for the problem with value $v_{2}(\theta^{*})$. Moreover, $\theta^{*}$ is optimal for the problem with value $u(S)$.\\
[0.2cm]
Hence, by the sufficient condition for optimality from Proposition \ref{sufficient-condition}, the pair $(\tau_{1}^{*}, \tau_{2}^{*})$ is optimal for $V(S)$.

\end{proof}

\section{Non-linear optimal $d$-stopping problem} \label{Section4}
Let $d \in \mathbbm{N}$ with $d \geq 2$. We consider the following non-linear $d$-stopping problem.\\
[0.3cm]
For $S \in \Theta$ given, 
\begin{equation}
V(S) \coloneqq \esssup_{(\tau_{1}, \tau_{2}, ..., \tau_{d}) \in \Theta_{S}^{d}} \rho_{S, \tau_{1} \vee \tau_{2} \vee ... \vee \tau_{d}}[\psi(\tau_{1}, \tau_{2}, ..., \tau_{d})].
\end{equation}

\begin{definition} \label{d-bi-admissible}
Let $(\psi(\tau_{1}, \tau_{2}, ..., \tau_{d}))_{(\tau_{1}, \tau_{2}, ..., \tau_{d}) \in \Theta^{d}}$ be a family of random variables indexed by $d$ Bermudan stopping times. We say that the family $\psi$ is \textbf{$d$-admissible} if it satisfies the following properties:
\begin{itemize}
\item[(a)] For each $(\tau_{1}, \tau_{2}, ..., \tau_{d}) \in \Theta^{d}$, the random variable $\psi(\tau_{1}, \tau_{2}, ..., \tau_{d})$ is $\mathcal{F}_{\tau_{1}\vee\tau_{2}\vee ... \vee \tau_{d}}$ - measurable.
\item[(b)] $\psi(\tau_{1}, \tau_{2}, ..., \tau_{d}) = \psi(\tilde{\tau}_{1}, \tilde{\tau}_{2}, ..., \tilde{\tau}_{d})$ on the set $\{\tau_{1} = \tilde{\tau}_{1}\} \cap \{\tau_{2} = \tilde{\tau}_{2}\} \cap ... \cap \{\tau_{d} = \tilde{\tau}_{d}\}$.
\end{itemize}
Moreover, for $p \in [1, +\infty]$ fixed, we say that a $d$-admissible family $\psi$ is $p$-integrable, if for all $\tau_{1} \in \Theta$, for all $\tau_{2} \in \Theta$, ..., for all $\tau_{d} \in \Theta$, $\psi(\tau_{1}, \tau_{2}, ..., \tau_{d}) \in L^{p}$.

\end{definition}

\noindent
In the sequel of this section, $\psi$ is assumed to be $d$-admissible $p$-integrable family.
\begin{assumption} \label{TEXAssumption3-3}
For each $S \in \Theta$, $V(S)$ is in $L^{p}$.

\end{assumption}

\begin{proposition}\label{TEXproposition5555}
i) $(V(S))_{S \in \Theta}$ is an admissible family.\\
[0.2cm]
ii) (maximising sequence) There exists a sequence of $d$ Bermudan stopping times $(\tau_{1}^{n}, \tau_{2}^{n}, ..., \tau_{d}^{n}) \in \Theta_{S}^{d}$, such that 
\begin{itemize}
\item{} $(\rho_{S, \tau_{1}^{n} \vee \tau_{2}^{n} \vee ... \vee \tau_{d}^{n}}[\psi(\tau_{1}^{n}, \tau_{2}^{n}, ..., \tau_{d}^{n})])_{n \in \mathbbm{N}}$ is non-decreasing, 
\item{} $V(S) = \lim_{n \to +\infty} \rho_{S, \tau_{1}^{n} \vee \tau_{2}^{n} \vee ... \vee \tau_{d}^{n}}[\psi(\tau_{1}^{n}, \tau_{2}^{n}, ..., \tau_{d}^{n})]$.
\end{itemize}
iii) The value family $(V(S))_{S \in \Theta}$ is a $(\Theta, \rho)$-supermartingale family.

\end{proposition}

\begin{property}\label{TEXPROperty-3-1}
Let $(\tau_{1}, ..., \tau_{d}) \in \Theta^{d}$, and $(\tau_{1}', ..., \tau_{d}') \in \Theta^{d}$. Let $A$ be in $\mathcal{F}_{\tau_{1}\wedge...\wedge\tau_{d}}$, and such that $\tau_{1} = \tau_{1}'$ on $A$, $...$, $\tau_{d} = \tau_{d}'$ on $A$. Then, 
$$\mathbbm{1}_{A}\rho_{\tau_{1} \wedge ... \wedge \tau_{d}, \tau_{1} \vee ... \vee \tau_{d}}[\psi(\tau_{1}, ..., \tau_{d})] = \mathbbm{1}_{A}\rho_{\tau_{1}' \wedge ... \wedge \tau_{d}', \tau_{1}' \vee ... \vee \tau_{d}'}[\psi(\tau_{1}', ..., \tau_{d}')].$$ 

\end{property}

\begin{proof}[Proof of Proposition \ref{TEXproposition5555}]
The proof is analogous to the proof of Proposition \ref{TEXproposition2}, and is given in the Appendix.\\
[0.2cm]

\end{proof}

\subsection{Reduction approach} \label{SubSection4-1}
As a first step, we reduce the non-linear optimal $d$-stopping problem (over a strategy with $d$ components) to $d$ auxiliary non-linear optimal stopping problems, each of which is an optimal $(d-1)$-stopping problem (that is, over a strategy with $(d-1)$ components). More precisely, for each $i \in \{1, 2, ..., d\}$, we define the following $d$ auxiliary problems:  for each $S \in \Theta$,
\begin{equation}
u^{(1)}(S) \coloneqq \esssup_{(\tau_{2}, ..., \tau_{d}) \in \Theta_{S}^{d-1}} \rho_{S, S \vee \tau_{2} \vee ... \vee \tau_{d}}[\psi(S, \tau_{2}, ..., \tau_{d})].
\end{equation}

\begin{equation}
u^{(2)}(S) \coloneqq \esssup_{(\tau_{1}, \tau_{3}..., \tau_{d}) \in \Theta_{S}^{d-1}} \rho_{S, \tau_{1} \vee S \vee \tau_{3}... \vee \tau_{d}}[\psi(\tau_{1}, S, \tau_{3}, ..., \tau_{d})].
\end{equation}
$$...$$
\begin{equation}
u^{(d)}(S) \coloneqq \esssup_{(\tau_{1}, \tau_{2}..., \tau_{d-1}) \in \Theta_{S}^{d-1}} \rho_{S, \tau_{1} \vee \tau_{2}... \vee \tau_{d-1} \vee S}[\psi(\tau_{1}, \tau_{2}, ..., \tau_{d-1}, S)].
\end{equation}
We define: for each $\tau \in \Theta$, 
\begin{equation}
\phi(\tau) \coloneqq u^{(1)}(\tau) \vee u^{(2)}(\tau) \vee ... \vee u^{(d)}(\tau).
\end{equation}
We now consider the non-linear single optimal stopping problem with this auxiliary reward (or pay-off) family, that is, for each $S \in \Theta$, we consider
\begin{equation} \label{TEXEquation1616}
u(S) \coloneqq \esssup_{\tau \in \Theta_{S}}\rho_{S, \tau}[\phi(\tau)].
\end{equation}

\begin{theorem}[Reduction]\label{TEXtheorem3333}
For each $S \in \Theta$, $u(S) = V(S)$.

\end{theorem}

\begin{remark} [Notation] 
In the particular case, where $d=2$, we used different notation. We have: $v_{1}(S) = u^{(2)}(S)$, $v_{2}(S) = u^{(1)}(S)$.

\end{remark}

\begin{remark} [Notation] 
For simplicity, we denote the vector of $(d-1)$ components $(\tau_{1}, ..., \tau_{i - 1}, \tau_{i+1}, ..., \tau_{d})$ by $\mathbf{\tau^{(-i)}}$.

\end{remark}

\begin{remark} 
Since $\psi$ is $d$-admissible, we have that for each $i \in \{1, ..., d\}$, the family $(\psi(\tau_{1}, ..., S, ..., \tau_{d}))_{\mathbf{\tau^{(-i)}} \in \Theta_{S}^{d-1}}$ is $(d-1)$-admissible. We can show by using Property \ref{TEXPROperty-3-1}, that the value family $(u^{(i)}(S))_{S \in \Theta}$ is an admissible family. Moreover, for each $i \in \{1, ..., d\}$, $\psi(S, S, ..., S) \leq u^{(i)}(S) \leq V(S)$ a.s. Hence, $u^{(i)}(S)$ is $p$-integrable by Assumption \ref{TEXAssumption3-3}.

\end{remark}

\begin{proof}[Proof of Theorem \ref{TEXtheorem3333}]
\textbf{Step 1:} We will show that $(V(S))_{S \in \Theta}$ is a $(\Theta, \rho)$-supermartingale family greater than or equal to $(\phi(S))_{S \in \Theta}$. We have already checked that $(V(S))$ is a $(\Theta, \rho)$-supermartingale family (cf. Proposition \ref{TEXproposition5555}). Moreover, by definition of $V(S)$, for any $i \in \{1, ..., d\}$, for any $(\tau_{1}, ..., \tau_{i-1}, \tau_{i+1}, ..., \tau_{d}) \in \Theta_{S}^{d-1}$, we have 
\begin{align*}
V(S) &\geq \rho_{S, \tau_{1} \vee ... \vee \tau_{i-1} \vee S \vee \tau_{i+1} \vee ... \vee \tau_{d}}[\psi(\tau_{1}, ..., \tau_{i-1}, S, \tau_{i+1}, ..., \tau_{d})]\\
&= \rho_{S, \tau_{1} \vee ... \vee \tau_{i-1} \vee \tau_{i+1} \vee ... \vee \tau_{d}}[\psi(\tau_{1}, ..., \tau_{i-1}, S, \tau_{i+1}, ..., \tau_{d})]
\end{align*}
Hence, we have
\begin{align*}
V(S) &\geq \esssup_{(\tau_{1}, ..., \tau_{i-1}, \tau_{i+1}, ..., \tau_{d}) \in \Theta_{S}^{d-1}} \rho_{S, \tau_{1} \vee ... \vee \tau_{i-1} \vee \tau_{i+1} \vee ... \vee \tau_{d}}[\psi(\tau_{1}, ..., \tau_{i-1}, S, \tau_{i+1}, ..., \tau_{d})]\\
&= u^{(i)}(S).
\end{align*}
As this holds true for all $i \in \{1, ..., d\}$, we get 
$$V(S) \geq u^{(1)}(S) \vee u^{(2)}(S) \vee ... \vee u^{(d)}(S) = \phi(S).$$
Hence, by the $(\Theta, \rho)$-Snell envelope characterisation of the value of the single non-linear optimal stopping problem \eqref{TEXEquation1616}, we get $V(S) \geq u(S)$.\\
[0.3cm]
\textbf{Step 2:} We will show the converse inequality.\\
[0.2cm]
Let $(\tau_{1}, \tau_{2}, ..., \tau_{d}) \in \Theta_{S}^{d}$. We will show that 
\begin{equation}
\rho_{S, \tau_{1} \vee \tau_{2} \vee ... \vee \tau_{d}}[\psi(\tau_{1}, \tau_{2}, ..., \tau_{d})] \leq \rho_{S, \tau_{1} \wedge \tau_{2} \wedge ... \wedge \tau_{d}}[\phi(\tau_{1} \wedge \tau_{2} \wedge ... \wedge \tau_{d})].
\end{equation}
There exists sets $(A_{i})_{i \in \{1, ..., d\}}$ with $\Omega = \bigcup_{i=1}^{d}A_{i}$, and $A_{i} \bigcap A_{j} = \emptyset$, for $i \neq j$, such that, for each $i \in \{1, ..., d\}$, $\tau_{1} \wedge ... \wedge \tau_{d} = \tau_{i}$ a.s. on $A_{i}$, and $A_{i} \in \mathcal{F}_{\tau_{1} \wedge ... \wedge \tau_{d}}$.\\
[0.2cm]
By the consistency property of $\rho$, we have 
\begin{align*}
\rho_{S, \tau_{1} \vee ... \vee \tau_{d}}[\psi(\tau_{1}, ..., \tau_{d})] &= \rho_{S, \tau_{1} \wedge ... \wedge \tau_{d}}[\rho_{\tau_{1} \wedge ... \wedge \tau_{d}, \tau_{1} \vee ... \vee \tau_{d}}[\psi(\tau_{1}, ..., \tau_{d})]]\\
&= \rho_{S, \tau_{1} \wedge ... \wedge \tau_{d}}[\sum^{d}_{i=1}\mathbbm{1}_{A_{i}}\rho_{\tau_{1} \wedge ... \wedge \tau_{d}, \tau_{1} \vee ... \vee \tau_{d}}[\psi(\tau_{1}, ..., \tau_{d})]].
\end{align*}
Let $i \in \{1, ..., d\}$. By property (ii) of $\rho$ (admissibility), we have:
\begin{align*}
\mathbbm{1}_{A_{i}}\rho_{\tau_{1} \wedge ... \wedge \tau_{d}, \tau_{1} \vee ... \vee \tau_{d}}[\psi(\tau_{1}, ..., \tau_{d})] &= \mathbbm{1}_{A_{i}}\rho_{\tau_{i}, \tau_{1} \vee ... \vee \tau_{i} \vee ... \vee \tau_{d}}[\psi(\tau_{1}, ..., \tau_{i}, ..., \tau_{d})]\\
&\leq \mathbbm{1}_{A_{i}}u^{(i)}(\tau_{i}) \leq \mathbbm{1}_{A_{i}}\phi(\tau_{i}),
\end{align*}
where we have used the also the definitions of $u^{(i)}(\tau_{i})$ and $\phi(\tau_{i})$ for the inequalities.\\
[0.2cm]
Hence, by using the monotonicity of $\rho$, we get
\begin{equation}
\rho_{S, \tau_{1} \vee ... \vee \tau_{d}}[\psi(\tau_{1}, ..., \tau_{d})] \leq \rho_{S, \tau_{1} \wedge ... \wedge \tau_{d}}[\sum_{i =1}^{d}\mathbbm{1}_{A_{i}}u^{(i)}(\tau_{i})] \leq \rho_{S, \tau_{1} \wedge ... \wedge \tau_{d}}[\sum_{i =1}^{d}\mathbbm{1}_{A_{i}}\phi(\tau_{i})].
\end{equation}
Now, 
\begin{align*}
\rho_{S, \tau_{1} \wedge ... \wedge \tau_{d}}[\sum_{i =1}^{d}\mathbbm{1}_{A_{i}}\phi(\tau_{i})] &= \rho_{S, \tau_{1} \wedge ... \wedge \tau_{d}}[\sum_{i}\phi(\tau_{1} \wedge ... \wedge \tau_{d})\mathbbm{1}_{A_{i}}]\\
&= \rho_{S, \tau_{1} \wedge ... \wedge \tau_{d}}[\phi(\tau_{1}\wedge ... \wedge \tau_{d})] \leq u(S),
\end{align*}
where we have used the properties on $(A_{i})_{i \in \{1, ..., d\}}$, and definition of $u(S)$ from \eqref{TEXEquation1616}.\\
[0.2cm]
As $(\tau_{1}, ..., \tau_{d})$ was taken arbitrary, we get $V(S) \leq u(S)$.

\end{proof}

\subsubsection{The particular symmetric case}
In this sub-section, we will consider the particular case where the pay-off family $(\psi(\tau_{1}, ..., \tau_{d}))$ is \textbf{symmetric} with respect to the individual components of the strategy, that is, $\psi(\tau_{1}, ..., \tau_{d}) = \psi(\tau_{\sigma(1)}, ..., \tau_{\sigma(d)})$, for any permutation $\sigma$ of the indices $\{1, ..., d\}$.

\begin{remark}
For the particular case, where $d = 2$, the pay-off family $\psi$ is symmetric, if and only if, $\psi(\tau_{1}, \tau_{2}) = \psi(\tau_{2}, \tau_{1})$. 

\end{remark}
\noindent
By symmetry, we can assume (without loss of generality) that $\tau_{1} \leq \tau_{2} \leq ... \leq \tau_{d}$.\\
[0.2cm]
We define: for each $S \in \Theta$, $\mathcal{O}_{S}^{d} \coloneqq \{(\tau_{1}, ..., \tau_{d}) \in \Theta_{S}^{d}: \tau_{1} \leq \tau_{2} \leq ... \leq \tau_{d}\}$. In the symmetric case, we have
\begin{align*}
V(S) &= \esssup_{(\tau_{1}, ..., \tau_{d}) \in \Theta_{S}^{d}}\rho_{S, \tau_{1} \vee ... \vee \tau_{d}}[\psi(\tau_{1}, ..., \tau_{d})]\\
&= \esssup_{(\tau_{1}, ..., \tau_{d}) \in \mathcal{O}_{S}^{d}}\rho_{S, \tau_{1} \vee ... \vee \tau_{d}}[\psi(\tau_{1}, ..., \tau_{d})]\\
&= \esssup_{(\tau_{1}, ..., \tau_{d}) \in \mathcal{O}_{S}^{d}}\rho_{S, \tau_{d}}[\psi(\tau_{1}, ..., \tau_{d})].
\end{align*}
It follows, from the symmetry assumption, that the auxiliary value families $u^{(i)}(S)$ coincide, that is, $u^{(1)}(S) = u^{(2)}(S) = ... = u^{(d)}(S)$, and hence, $\phi(S) = u^{(1)}(S)$. Moreover, we can write:
\begin{align*}
u^{(1)}(S) &= \esssup_{(\tau_{2}, ..., \tau_{d}) \in \mathcal{O}_{S}^{d-1}} \rho_{S, \tau_{2} \vee ... \vee \tau_{d}}[\psi(S, \tau_{2}, ..., \tau_{d})]\\
&= \esssup_{(\tau_{2}, ..., \tau_{d}) \in \mathcal{O}_{S}^{d-1}} \rho_{S, \tau_{d}}[\psi(S, \tau_{2}, ..., \tau_{d})].
\end{align*}
Thus, the Reduction Theorem \ref{TEXtheorem3333} can be expressed as follows:
$$V(S) = \esssup_{\tau \in \Theta_{S}} \rho_{S, \tau}[u^{(1)}(\tau)].$$
We can continue the computations by using again the symmetry and the Reduction Theorem \ref{TEXtheorem3333} (applied to $u^{(1)}(\tau)$). Indeed, the non-linear optimal stopping problem with value $u^{(1)}(\tau)$ is a symmetric $(d-1)$-stopping problem.\\
[0.2cm]
Let us consider $S_{1}, S_{2} \in \Theta$, such that $S_{1} \leq S_{2}$ and the new reward
$$\phi_{2}(S_{1}, S_{2}) \coloneqq \esssup_{(\tau_{3}, ..., \tau_{d}) \in \mathcal{O}_{S_{2}}^{d-2}} \rho_{S_{2}, \tau_{3} \vee ... \vee \tau_{d}}[\psi(S_{1}, S_{2}, \tau_{3}, ..., \tau_{d})].$$
The Reduction Theorem \ref{TEXtheorem3333} and the symmetry give us that 
$$u^{(1)}(S_{1}) = \esssup_{\tau \in \Theta_{S_{1}}}\rho_{S_{1}, \tau}[\phi_{2}(S_{1}, \tau)].$$
We will then consider the non-linear optimal stopping problem with value family $\phi_{2}(S_{1}, \tau)$, which is a symmetric $(d-2)$-stopping problem, and so forth.\\
[0.2cm]
By induction, we define the new award: 
\begin{equation} \label{TEXEquation1919}
\phi_{i}(S_{1}, ..., S_{i}) \coloneqq \esssup_{(\tau_{i+1}, ..., \tau_{d}) \in \mathcal{O}^{d-i}_{S_{i}}}\rho_{S_{i}, \tau_{i+1} \vee ... \vee \tau_{d}}[\psi(S_{1}, ..., S_{i}, \tau_{i+1}, ..., \tau_{d})],
\end{equation}
for each $(S_{1}, ..., S_{i}) \in \mathcal{O}^{i}_{S}$.

\begin{proposition} \label{TEXPROPosition6666}
Let $\psi$ be a symmetric $d$-admissible, $p$-integrable family. We define the family $\phi_{d-1}$ by:
\begin{equation}\label{TEXequation3-18}
\phi_{d-1}(S_{1}, ..., S_{d-1}) \coloneqq \esssup_{\tau \in \Theta_{S_{d-1}}}\rho_{S_{d-1}, \tau}[\psi(S_{1}, ..., S_{d-1}, \tau)].
\end{equation}
Moreover, we define, by backward induction, the following families: for $i \in \{d-2, ..., 1\}$, for $(S_{1}, ..., S_{i}) \in \mathcal{O}_{S}^{i}$,
\begin{equation} \label{TEXEquation2121}
\phi_{i}(S_{1}, ..., S_{i}) \coloneqq \esssup_{\tau \in \Theta_{S_{i}}} \rho_{S_{i}, \tau}[\phi_{i+1}(S_{1}, ..., S_{i}, \tau)].
\end{equation}
Then, we have 
$$V(S) = \esssup_{\tau \in \Theta_{S}}\rho_{S, \tau}[\phi_{1}(\tau)].$$

\end{proposition}

\begin{proof}
We know, by symmetry and by the Reduction Theorem \ref{TEXtheorem3333}, that 
\begin{align*}
\phi_{i}(S_{1}, ..., S_{i}) &= \esssup_{\tau \in \Theta_{S_{i}}}\rho_{S_{i}, \tau}[\phi_{i+1}(S_{1}, ..., S_{i}, \tau)]\\
&= \esssup_{(\tau_{i+1}, ..., \tau_{d}) \in \mathcal{O}^{d-i}_{S_{i}}}\rho_{S_{i}, \tau_{i+1} \vee ... \vee \tau_{d}}[\psi(S_{1}, ..., S_{i}, \tau_{i+1}, ..., \tau_{d})].
\end{align*}
We prove the result by backward induction relying on the symmetry and on the Reduction Theorem \ref{TEXtheorem3333}. We have, by definition of $\phi_{d-1}(S_{1}, ..., S_{d-1})$,
$$\phi_{d-1}(S_{1}, ..., S_{d-1}) = \esssup_{\tau \in \Theta_{S_{d-1}}}\rho_{S_{d-1}, \tau}[\psi(S_{1}, ..., S_{d-1}, \tau)].$$
We suppose, by backward induction, that
\begin{align*}
\phi_{i+1}(S_{1}, ..., S_{i+1}) = \esssup_{(\tau_{i+2}, ..., \tau_{d}) \in \mathcal{O}_{S_{i+1}}^{d-(i+1)}} \rho_{S_{i+1}, \tau_{i+2} \vee ... \vee \tau_{d}}[\psi(S_{1}, ..., S_{i+1}, \tau_{i+2}, ..., \tau_{d})].
\end{align*}
We will show this property at rank $i$. If we do so, then, we will conclude by the symmetry and by the Reduction Theorem \ref{TEXtheorem3333}. We have, by Eq. \eqref{TEXEquation2121},
\begin{align*}
\phi_{i}(S_{1}, ..., S_{i}) = \esssup_{\tau \in \Theta_{S_{i}}}\rho_{S_{i}, \tau}[\phi_{i+1}(S_{1}, ..., S_{i}, \tau)].
\end{align*}
By replacing, we get
\begin{equation}\label{TEXEquation222222}
\begin{aligned}
&\phi_{i}(S_{1}, ..., S_{i})\\
 &= \esssup_{\tau \in \Theta_{S_{i}}}\rho_{S_{i}, \tau}[\esssup_{(\tau_{i+2}, ..., \tau_{d}) \in \mathcal{O}_{\tau}^{d-(i+1)}}\rho_{\tau, \tau_{i+2} \vee ... \vee \tau_{d}}[\psi(S_{1}, ..., S_{i}, \tau, \tau_{i+2}, ..., \tau_{d})]]
\end{aligned}
\end{equation}
We define:
\begin{equation}
u_{i}(S_{1}, ..., S_{i}) \coloneqq \esssup_{\tau_{i+1}, ..., \tau_{d} \in \mathcal{O}_{S_{i}}^{d-i}}\rho_{S_{i}, \tau_{i+1} \vee ... \vee \tau_{d}}[\psi(S_{1}, ..., S_{i}, \tau_{i+1}, ..., \tau_{d})].
\end{equation}
We will show that $u_{i}(S_{1}, ..., S_{i}) = \phi_{i}(S_{1}, ..., S_{i})$.\\
[0.2cm]
By the Reduction Theorem \ref{TEXtheorem3333} and the symmetry,
\begin{equation}\label{TEXEquation232323}
\begin{aligned}
&u_{i}(S_{1}, ..., S_{i})\\ 
&= \esssup_{\tau \in \Theta_{S_{i}}}\rho _{S_{i}, \tau}[\esssup_{(\tau_{i+2}, ..., \tau_{d}) \in \mathcal{O}^{d-(i+1)}_{\tau}}\rho_{\tau, \tau_{i+2} \vee ... \vee \tau_{d}}[\psi(S_{1}, ..., S_{i}, \tau, \tau_{i+2}, ..., \tau_{d})]].
\end{aligned}
\end{equation}
From Eq. \eqref{TEXEquation222222} and \eqref{TEXEquation232323}, we get 
$$\phi_{i}(S_{1}, ..., S_{i}) = u_{i}(S_{1}, ..., S_{i}),$$
which is what we wanted to show.

\end{proof}

\paragraph{Some particular examples of symmetric rewards}
\begin{remark}
A well-known example of a symmetric reward family is the \textbf{additive reward}:
$$\psi(\tau_{1}, ..., \tau_{d}) = \eta(\tau_{1}) + \eta(\tau_{2}) + ... + \eta(\tau_{d}),$$
where $\eta = (\eta(\tau))_{\tau \in \Theta}$ is an admissible, $p$-integrable family of random variables. In the particular case, where $\rho_{S, \tau} = \mathbbm{E}[\cdot|\mathcal{F}_{S}]$, the optimisation problem simplifies and a direct study of $V(S)$ gives that $V(S) = d\bar{v}(S)$, where $\bar{v}(S) \coloneqq \esssup_{\tau \in \Theta_{S}} \mathbbm{E}[\eta(\tau)|\mathcal{F}_{S}]$ (cf. \cite{Kobylanski-Quenez-Mironescu-2}). However, in the case where $\rho_{S, \tau}$ is non-linear, we do not have such a simple expression even in the case where $\rho_{S, \tau}$ is assumed to be translation invariant.

\end{remark}

\noindent
If we assume moreover, that $\rho$ is \textbf{translation invariant}\footnote{We say that $\rho$ is \textbf{translation invariant}, if for all $S, \tau \in \Theta$, for all $L \in L^{p}(\mathcal{F}_{S})$, $\rho_{S, \tau}[\eta + L] = \rho_{S, \tau}[\eta] + L$.}, we get (cf. Eq. \eqref{TEXequation3-18} from Proposition \ref{TEXPROPosition6666}),
\begin{align*}
\phi_{d-1}(S_{1}, ..., S_{d-1}) &= \esssup_{\tau \in \Theta_{S_{d-1}}}\rho_{S_{d-1}, \tau}[\eta(S_{1}) + ... +\eta(S_{d-1}) + \eta(\tau)]\\
&= \esssup_{\tau \in \Theta_{S_{d-1}}} \eta(S_{1}) + ... + \eta(S_{d-1}) + \rho_{S_{d-1}, \tau}[\eta(\tau)]\\
&= \sum_{i =1}^{d-1}\eta(S_{i}) + \esssup_{\tau \in \Theta_{S_{d-1}}}\rho_{S_{d-1}, \tau}[\eta(\tau)].
\end{align*}
For $\bar{\tau} \in \Theta$, let us denote by $\bar{v}(\bar{\tau})$ the following non-linear \emph{single} optimal stopping problem:
$$\bar{v}_{d-1}(\bar{\tau}) \coloneqq \esssup_{\tau \in \Theta_{\bar{\tau}}}\rho_{\bar{\tau}, \tau}[\eta(\tau)].$$
Then, we have:
\begin{align*}
\phi_{d-1}(S_{1}, ..., S_{d-1}) = \sum_{i =1}^{d-1}\eta(S_{i}) + \bar{v}_{d-1}(S_{d-1}).
\end{align*}
By the definition \eqref{TEXEquation2121} from Proposition \ref{TEXPROPosition6666},
\begin{align*}
\phi_{d-2}(S_{1}, ..., S_{d-2}) &= \esssup_{\tau \in \Theta_{S_{d-2}}}\rho_{S_{d-2}, \tau}[\phi_{d-1}(S_{1}, ..., S_{d-2}, \tau)]\\
&= \esssup_{\tau \in \Theta_{S_{d-2}}}\rho_{S_{d-2}, \tau}[\sum_{i=1}^{d-2}\eta(S_{i}) + \eta(\tau) + \bar{v}_{d-1}(\tau)]\\
&= \sum_{i=1}^{d-2}\eta(S_{i}) + \esssup_{\tau \in \Theta_{S_{d-2}}}\rho_{S_{d-2}, \tau}[\eta(\tau) + \bar{v}_{d-1}(\tau)]\\
&= \sum_{i=1}^{d-2}\eta(S_{i}) + \bar{v}_{d-2}(S_{d-2}),
\end{align*}
where we have used also the translation invariance of $\rho$ for the last but one equality, and where $\bar{v}_{d-2}(S_{d-2}) \coloneqq \esssup_{\tau \in \Theta_{S_{d-2}}}\rho_{S_{d-2}, \tau}[\eta(\tau) + \bar{v}_{d-1}(\tau)]$.\\
[0.2cm]
By induction, we have:
\begin{align*}
\phi_{i}(S_{1}, ..., S_{i}) = \sum_{k=1}^{i}\eta(S_{k}) + \bar{v}_{i}(S_{i}),
\end{align*}
where $\bar{v}_{i}(S_{i}) \coloneqq \esssup_{\tau \in \Theta_{S_{i}}}\rho_{S_{i}, \tau}[\eta(S_{i+1}) + \bar{v}_{i+1}(\tau)]$.\\[0.2cm]
By Proposition \ref{TEXPROPosition6666}, $V(S) = \esssup_{\tau \in \Theta_{S}}\rho_{S, \tau}[\phi_{1}(\tau)]$. We have thus brought the non-linear $d$-optimal stopping problem to $d$ non-linear single optimal stopping problems (nested in each other as explained above).
\begin{remark}
We can consider the case where there exists $n$ (not depending on $\omega$) such that $\theta_{k}(\omega) = T$, for all $k \geq n$. The particular case of \textbf{swing options} enters into the above additive framework. In this case, the reward is additive (as before), and the stopping times of the sequence $(\theta_{0}, \theta_{1}, ..., \theta_{n})$ are equi-distant from each other by the fixed distance $\delta > 0$. The distance $\delta > 0$ is often referred to as the refracting time. More precisely, we have $\theta_{0} = 0, \theta_{1} = \delta, ..., \theta_{n} = n\delta = T$.

\end{remark}
\noindent
Another example is the example of the \textbf{multiplicative reward}, that is:
\begin{align*}
\psi(\tau_{1}, ..., \tau_{d}) = \eta(\tau_{1}) \times ... \times \eta(\tau_{d}),
\end{align*}
where $\eta = (\eta(\tau))_{\tau \in \Theta}$ is an admissible, $p$-integrable family of random variables. In this case, if we assume that $\eta(\tau) \geq 0$, for each $\tau \in \Theta$, and that $\rho$ is \textbf{positively homogeneous}\footnote{We say that $\rho$ is positively homogeneous, if, for all $S, \tau \in \Theta$, for all non-negative $L \in L^{p}(\mathcal{F}_{S})$, $\rho_{S, \tau}(L\eta) = L\rho_{S, \tau}(\eta)$.}, we get:
\begin{align*}
\phi_{d-1}(S_{1}, ..., S_{d-1}) &= \esssup_{\tau \in S_{d-1}}\rho_{S_{d-1}, \tau}[\eta(S_{1}) \times ... \times \eta(S_{d-1}) \times \eta(\tau)]\\
&= \eta(S_{1}) \times ... \times \eta(S_{d-1}) \times \esssup_{\tau \in S_{d-1}}\rho_{S_{d-1}, \tau}[\eta(\tau)]\\
&= \eta(S_{1}) \times ... \times \eta(S_{d-1}) \times \tilde{v}_{d-1}(S_{d-1}),
\end{align*}
where we have used the positive homogeneity of $\rho$, and where $\tilde{v}_{d-1}(S_{d-1}) \coloneqq \esssup_{\tau \in S_{d-1}}\rho_{S_{d-1}, \tau}[\eta(\tau)]$.\\
[0.2cm]
Moreover, by \eqref{TEXEquation2121}
\begin{align*}
\phi_{d-2}(S_{1}, ..., S_{d-2}) &= \esssup_{\tau \in \Theta_{S_{d-2}}}\rho_{S_{d-2}, \tau}[\phi_{d-1}(S_{1}, ..., S_{d-2}, \tau)]\\
&=  \esssup_{\tau \in \Theta_{S_{d-2}}}\rho_{S_{d-2}, \tau}[\eta(S_{1}) \times ... \times \eta(S_{d-2}) \times \tilde{v}_{d-1}(\tau)].
\end{align*}
By using again the positive homogeneity, we get 
\begin{align*}
\phi_{d-2}(S_{1}, ..., S_{d-2}) &= \eta(S_{1}) \times ... \times \eta(S_{d-2}) \times \esssup_{\tau \in \Theta_{S_{d-2}}}\rho_{S_{d-2}, \tau}[\eta(\tau)\tilde{v}_{d-1}(\tau)]\\
&= \eta(S_{1}) \times ... \times \eta(S_{d-2}) \times \tilde{v}_{d-2}(S_{d-2}),
\end{align*}
where $\tilde{v}_{d-2}(S_{d-2}) \coloneqq \esssup_{\tau \in \Theta_{S_{d-2}}}\rho_{S_{d-2}, \tau}[\eta(\tau)\tilde{v}_{d-1}(\tau)]$.\\[0.2cm]
By induction, we have:
\begin{equation}
\phi_{i}(S_{1}, ..., S_{i}) = \eta(S_{1}) \times ... \times \eta(S_{i}) \times \tilde{v}_{i}(S_{i}),
\end{equation}
where $\tilde{v}_{i}(S_{i}) \coloneqq \esssup_{\tau \in \Theta_{S_{i}}} \rho_{S_{i}, \tau}[\eta(\tau)\tilde{v}_{i+1}(\tau)]$.\\
[0.2cm]
By Proposition \ref{TEXPROPosition6666}, we get: $V(S) = \esssup_{\tau \in \Theta_{S}}\rho_{S, \tau}[\phi_{1}(\tau)]$. This reasoning illustrates that, in the case of a multiplicative reward, we can solve the non-linear optimal multiple  stopping problem by solving successively non-linear single optimal stopping problems. 

\subsection{Optimal stopping times in the non-linear optimal $d$-stopping problem} \label{SubSection4-2}
\begin{proposition}[Construction of optimal stopping times for the non-linear optimal $d$-stopping problem. Sufficient condition] We assume that:\\
[0.2cm]
i) $\theta^{*} \in \Theta_{S}$ is optimal for the problem with value $u(S)$.\\
[0.2cm]
ii) For each $i \in \{1, ..., d\}$, the vector $\theta^{(-i), *} \coloneqq (\theta_{1}^{(-i), *}, ..., \theta_{i-1}^{(-i), *}, \theta_{i+1}^{(-i), *}, ...,\theta_{d}^{(-i), *}) \in \Theta^{d-1}_{\theta^{*}}$ is optimal for the problem with value $u^{(i)}(\theta^{*})$\\
[0.2cm]
Let $B_{1}, ..., B_{d}$ be a partition of $\Omega$ such that for each $i \in \{1, ..., d\}$, $B_{i}$ is $\mathcal{F}_{\theta^{*}}$-measurable and $\phi(\theta^{*}) = u^{(i)}(\theta^{*})$ a.s. on $B_{i}$. Then, the vector of $d$ Bermudan stopping times $(\tau_{1}^{*}, ..., \tau_{d}^{*})$ defined by:\\
[0.2cm]
$\tau_{1}^{*} = \theta^{*}\mathbbm{1}_{B_{1}} + \sum^{d}_{i=2}\theta_{1}^{(-i), *}\mathbbm{1}_{B_{i}}$, where $\theta_{1}^{(-i), *}$ denotes the first component of the vector $\theta^{(-i), *}$;\\
[0.2cm]
$\tau_{2}^{*} = \theta^{*}\mathbbm{1}_{B_{2}} + \sum^{d}_{i=1, i\neq2}\theta_{2}^{(-i), *}\mathbbm{1}_{B_{i}}$, where $\theta_{2}^{(-i), *}$ denotes the second component of the vector $\theta^{(-i), *}$;\\
[0.2cm]
...\\
[0.2cm]
$\tau_{d}^{*} = \theta^{*}\mathbbm{1}_{B_{d}} + \sum^{d-1}_{i=1}\theta_{d}^{(-i), *}\mathbbm{1}_{B_{i}}$, where $\theta_{d}^{(-i), *}$ denotes the final component of the vector $\theta^{(-i), *}$, is optimal for the non-linear optimal $d$-stopping problem.

\end{proposition}

\begin{proof}
We have $\tau_{1}^{*} \wedge \tau_{2}^{*} \wedge ... \wedge \tau_{d}^{*} = \theta^{*}$ (as ($B_{1}, ..., B_{d}$) is a partition and the elements of $\theta^{(-i), *}$ are in $\Theta_{\theta^{*}}$). Moreover, on $B_{i}$, $\tau_{1}^{*} \wedge \tau_{2}^{*} \wedge ... \wedge \tau_{d}^{*} = \theta^{*} = \tau^{*}_{i}$ (where we have used the definition of $\tau_{i}^{*}$), and on $B_{i}$, $\tau_{1}^{*} \vee \tau_{2}^{*} \vee ... \vee \tau_{d}^{*} = \theta^{(-i), *}_{1} \vee ... \vee \theta^{(-i), *}_{i-1} \vee \theta^{*} \vee \theta^{(-i), *}_{i+1} \vee ... \vee \theta^{(-i), *}_{d} = \theta^{(-i), *}_{1} \vee ... \vee \theta^{(-i), *}_{i-1} \vee \theta^{(-i), *}_{i+1} \vee ... \vee \theta^{(-i), *}_{d} = \tau^{*}_{1} \vee ... \vee \tau^{*}_{i-1} \vee \tau^{*}_{i+1} \vee ... \vee \tau^{*}_{d}$.\\
[0.2cm]
We have, by i), and by definition of $(B_{1}, ..., B_{d})$,
\begin{align*}
u(S) = \rho_{S, \theta^{*}}[\phi(\theta^{*})] &= \rho_{S, \theta^{*}}[\sum^{d}_{i=1}\mathbbm{1}_{B_{i}}\phi(\theta^{*})]\\
&= \rho_{S, \theta^{*}}[\sum^{d}_{i=1}\mathbbm{1}_{B_{i}}u^{(i)}(\theta^{*})] = \rho_{S, \tau_{1}^{*} \wedge \tau_{2}^{*} \wedge ... \wedge \tau_{d}^{*}}[\sum^{d}_{i=1}\mathbbm{1}_{B_{i}}u^{(i)}(\theta^{*})].
\end{align*}
Now, by ii), we have for each $i \in \{1, ..., d\}$,
\begin{align*}
&\mathbbm{1}_{B_{i}}u^{(i)}(\theta^{*})\\ 
&= \mathbbm{1}_{B_{i}}\rho_{\theta^{*}, \theta_{1}^{(-i), *} \vee ... \vee \theta_{i-1}^{(-i), *} \vee \theta^{*} \vee \theta_{i+1}^{(-i), *} \vee ... \vee \theta_{d}^{(-i), *}}[\psi(\theta_{1}^{(-i), *}, ..., \theta_{i-1}^{(-i), *}, \theta^{*}, \theta_{i+1}^{(-i), *}, ..., \theta_{d}^{(-i), *})]\\
&= \mathbbm{1}_{B_{i}}\rho_{\tau^{*}_{1} \wedge ... \wedge \tau^{*}_{d}, \tau^{*}_{1} \vee ... \vee \tau^{*}_{d}}[\psi(\tau^{*}_{1}, ..., \tau^{*}_{i-1}, \tau^{*}_{i}, \tau^{*}_{i+1}, ..., \tau^{*}_{d})],
\end{align*}
where we have used the definition of $\tau^{*}_{i}$ and Property \ref{TEXPROperty-3-1}.\\
[0.2cm]
By putting everything together, we get:
\begin{align*}
u(S) &=  \rho_{S, \tau_{1}^{*} \wedge ... \wedge \tau_{d}^{*}}[\sum^{d}_{i=1}\mathbbm{1}_{B_{i}}\rho_{\tau^{*}_{1} \wedge ... \wedge \tau^{*}_{d}, \tau^{*}_{1} \vee ... \vee \tau^{*}_{d}}[\psi(\tau^{*}_{1}, ..., \tau^{*}_{i-1}, \tau^{*}_{i}, \tau^{*}_{i+1}, ..., \tau^{*}_{d})]]\\
&= \rho_{S, \tau_{1}^{*} \wedge ... \wedge \tau_{d}^{*}}[\rho_{\tau^{*}_{1} \wedge ... \wedge \tau^{*}_{d}, \tau^{*}_{1} \vee ... \vee \tau^{*}_{d}}[\psi(\tau^{*}_{1}, ..., \tau^{*}_{i-1}, \tau^{*}_{i}, \tau^{*}_{i+1}, ..., \tau^{*}_{d})]]\\
&= \rho_{S, \tau_{1}^{*} \vee ... \vee \tau_{d}^{*}}[\psi(\tau_{1}^{*}, ..., \tau_{d}^{*})],
\end{align*}
where we have used the time consistency of $\rho$ for the last equality.\\
[0.2cm]
We conclude that $(\tau^{*}_{1}, ..., \tau^{*}_{d})$ is optimal for $u(S)$. By Theorem \ref{TEXtheorem3333}, $u(S) = V(S)$, hence $(\tau^{*}_{1}, ..., \tau^{*}_{d})$ is optimal for $V(S)$.

\end{proof}

\begin{proposition}[A necessary condition for optimality in the non-linear optimal $d$-stopping problem]\label{multiple-necessary-condition}
Let $S \in \Theta$. Suppose that a given $d$-uple $(\tau^{*}_{1}, ..., \tau^{*}_{d})$ is optimal for $V(S)$. Let $(A_{i})_{i \in \{1, ..., d\}}$ be a partition of $\Omega$ such that, for each $i \in \{1, ..., d\}$, $A_{i}$ is $\mathcal{F}_{\tau^{*}_{1} \wedge ...\wedge \tau^{*}_{d}}$-measurable, and such that on $A_{i}$, $\tau^{*}_{i} = \tau^{*}_{1} \wedge ... \wedge \tau^{*}_{d}$. Then the following holds:\\
[0.2cm]
i) $\tau^{*}_{1} \wedge ... \wedge \tau^{*}_{d}$ is optimal for the problem with value $u(S)$.\\
[0.2cm]
Moreover, if $\rho$ is \textbf{strictly monotone}, then:\\
[0.2cm]
ii) For each $i \in \{1, ..., d\}$, $\theta^{(-i), *}$, defined by, $\theta^{(-i), *} \coloneqq\\ 
(\theta_{1}^{(-i), *}, ..., \theta_{i-1}^{(-i), *}, \theta_{i+1}^{(-i), *}, ...,\theta_{d}^{(-i), *}) \coloneqq (\tau^{*}_{1}, ..., \tau^{*}_{i-1}, \tau^{*}_{i+1}, ..., \tau^{*}_{d})$, is optimal for $u^{(i)}(\tau^{*}_{1} \wedge ... \wedge \tau^{*}_{d})$ on $A_{i}$.

\end{proposition}

\begin{proof}
We recall: 
\begin{equation}\label{TEXEQuation-3-2424}
V(S) = \rho_{S, \tau_{1}^{*} \vee ... \vee \tau_{d}^{*}}[\psi(\tau_{1}^{*}, ..., \tau_{d}^{*})] = u(S),
\end{equation} 
(due to the optimality of $(\tau_{1}^{*}, ..., \tau_{d}^{*})$ for $V(S)$ and to Theorem \ref{TEXtheorem3333}). By following the same arguments as in Step 2 of the proof of Theorem \ref{TEXtheorem3333}, with $\tau_{i}^{*}$ in place of $\tau_{i}$ (for each $i \in \{1, ..., d\}$), we get:
\begin{equation}\label{TEXEQuation-3-2525}
\begin{aligned}
\rho_{S, \tau_{1}^{*} \vee ... \vee \tau_{d}^{*}}[\psi(\tau_{1}^{*}, ..., \tau_{d}^{*})] &= \rho_{S, \tau^{*}_{1} \wedge ... \wedge \tau^{*}_{d}}[\sum^{d}_{i=1}\mathbbm{1}_{A^{i}}\rho_{\tau^{*}_{1} \wedge ... \wedge \tau^{*}_{d}, \tau^{*}_{1} \vee ... \vee \tau^{*}_{d}}[\psi(\tau^{*}_{1}, ..., \tau^{*}_{d})]]\\
&\leq \rho_{S, \tau^{*}_{1} \wedge ... \wedge \tau^{*}_{d}}[\sum^{d}_{i=1}\mathbbm{1}_{A^{i}}u^{(i)}(\tau^{*}_{i})]\\
&\leq \rho_{S, \tau^{*}_{1} \wedge ... \wedge \tau^{*}_{d}}[\sum^{d}_{i=1}\mathbbm{1}_{A^{i}}\phi(\tau^{*}_{i})]\\
&= \rho_{S, \tau_{1}^{*} \wedge ... \wedge \tau_{d}^{*}}[\phi(\tau_{1}^{*} \wedge ... \wedge \tau_{d}^{*})] \leq u(S).
\end{aligned}
\end{equation}
By Eq.\eqref{TEXEQuation-3-2424}, for the LHS of Eq.\eqref{TEXEQuation-3-2525} we have
$$\rho_{S, \tau_{1}^{*} \vee ... \vee \tau_{d}^{*}}[\psi(\tau_{1}^{*}, ..., \tau_{d}^{*})] = u(S).$$
Hence, all the inequalities in Eq.\eqref{TEXEQuation-3-2525} are equalities. Hence, $\tau_{1}^{*} \wedge ...\wedge \tau_{d}^{*}$ is optimal for $u(S)$. Moreover, 
\begin{align*}
\rho_{S, \tau^{*}_{1} \wedge ... \wedge \tau^{*}_{d}}[\sum^{d}_{i=1}\mathbbm{1}_{A^{i}}\rho_{\tau^{*}_{1} \wedge ... \wedge \tau^{*}_{d}, \tau^{*}_{1} \vee ... \vee \tau^{*}_{d}}[\psi(\tau^{*}_{1}, ..., \tau^{*}_{d})]] = \rho_{S, \tau^{*}_{1} \wedge ... \wedge \tau^{*}_{d}}[\sum^{d}_{i=1}\mathbbm{1}_{A^{i}}u^{(i)}(\tau^{*}_{i})].
\end{align*}
Furthermore, by definition of $u^{(i)}(\tau^{*}_{i})$ and by definition of $A^{i}$, we have:
\begin{align*}
\sum^{d}_{i=1}\mathbbm{1}_{A^{i}}\rho_{\tau^{*}_{1} \wedge ... \wedge \tau^{*}_{d}, \tau^{*}_{1} \vee ... \vee \tau^{*}_{d}}[\psi(\tau^{*}_{1}, ..., \tau^{*}_{d})] &\leq \sum^{d}_{i=1}\mathbbm{1}_{A^{i}}u^{(i)}(\tau^{*}_{i}) = \sum^{d}_{i=1}\mathbbm{1}_{A^{i}}u^{(i)}(\tau^{*}_{1} \wedge ... \wedge \tau^{*}_{d}).
\end{align*}
Hence, under the additional assumption that $\rho_{S, \tau^{*}_{1} \wedge ... \wedge \tau^{*}_{d}}$ is strictly monotone, we get:
$$\sum^{d}_{i=1}\mathbbm{1}_{A^{i}}\rho_{\tau^{*}_{1} \wedge ... \wedge \tau^{*}_{d}, \tau^{*}_{1} \vee ... \vee \tau^{*}_{d}}[\psi(\tau^{*}_{1}, ..., \tau^{*}_{d})] = \sum^{d}_{i=1}\mathbbm{1}_{A^{i}}u^{(i)}(\tau^{*}_{1} \wedge ... \wedge \tau^{*}_{d}).$$
Hence, statement ii) is proven.

\end{proof}

\section{Examples of non-linear operators}\label{Section5}

\subsection{Non-linear $g$-evaluations induced by BSDEs}
Let $(\Omega, \mathcal{F}, (\mathcal{F}_{t})_{t \in [0, T]}, P)$ be a filtered probability space, satisfying the usual conditions, and such that $(\mathcal{F}_{t})_{t \in [0, T]}$ is the natural completed filtration of a one-dimensional Brownian motion $W = (W_{t})_{t \in [0, T]}$. Let $p = 2$. We consider the operator $\rho_{S, \tau}[\cdot] = \mathcal{E}^{g}_{S, \tau}[\cdot]$, where $\mathcal{E}^{g}_{S, \tau}[\cdot]: L^{2}(\mathcal{F}_{\tau}) \to L^{2}(\mathcal{F}_{S})$ is the non-linear $g$-evaluation induced by the Backward Stochastic Differential Equation (BSDE) with standard Lipschitz driver $g$: 
$$y_{S} = \eta + \int^{\tau}_{S}g(u, y_{u}, z_{u})du - \int^{\tau}_{S}z_{u}dW_{u},$$
where $\eta \in L^{2}(\mathcal{F}_{\tau})$ and $\mathcal{E}_{S, \tau}^{g}[\eta] \coloneqq y_{S}$.\\
[0.2cm]
The operators $\rho_{S, \tau}[\cdot] = \mathcal{E}_{S, \tau}^{g}[\cdot]$ satisfy the properties: (i) (with $p=2$), (ii) (admissibility), (iii) (knowledge preserving property), (iv) (monotonicity) (without any further assumptions on $g$ as we placed ourselves in the Brownian framework for this example), (v) (consistency), (vi) (``generalized zero-one law''), (vii) (monotone Fatou property with respect to terminal condition) (which is even true with equality, in place of the inequality, due to the property of continuity with respect to the terminal condition), (viii) and (ix) (cf. Remark 3.3 from our paper \cite{Grigorova-Quenez-Peng-1}). Moreover, we note that $\rho_{S, \tau}[\cdot]$ is strictly monotone without any additional assumptions on $g$ (a property that we use in Proposition \ref{double-necessary-condition} and in Proposition \ref{multiple-necessary-condition}), as we placed ourselves in the Brownian framework. It remains to check that properties (vi bis) and (vi ter) hold true (as well as Property \ref{TEXPROperty-3-1} from Subsection \ref{SubSection4-2}, generalizing property (vi ter) to tackle the $d$-stopping case).\\
[0.2cm]
Let us first check (vi bis). We have:
\begin{align*}
\mathbbm{1}_{\{\tau_{1} = \tau_{1}'\}}\mathcal{E}^{g}_{\tau_{1}, \tau_{1}\vee\tau_{2}}[\psi(\tau_{1}, \tau_{2})] = \mathcal{E}^{g \mathbbm{1}_{\{\tau_{1} = \tau_{1}'\}}}_{\tau_{1}, \tau_{1}\vee\tau_{2}}[\mathbbm{1}_{\{\tau_{1} = \tau_{1}'\}}\psi(\tau_{1}, \tau_{2})],
\end{align*}
where we have used the convention in the notation from \cite{Grigorova-2}.\\
[0.2cm] 
By bi-admissibility of the family $\psi$, we have,\\ 
$\mathbbm{1}_{\{\tau_{1} = \tau_{1}'\}}\psi(\tau_{1}, \tau_{2}) = \mathbbm{1}_{\{\tau_{1} = \tau_{1}'\}}\psi(\tau_{1}', \tau_{2})$. Hence, 
\begin{align*}
\mathcal{E}^{g \mathbbm{1}_{\{\tau_{1} = \tau_{1}'\}}}_{\tau_{1}, \tau_{1}\vee\tau_{2}}[\mathbbm{1}_{\{\tau_{1} = \tau_{1}'\}}\psi(\tau_{1}, \tau_{2})] &= \mathcal{E}^{g^{\tau_{1}'\vee\tau_{2}}\mathbbm{1}_{\{\tau_{1}=\tau_{1}'\}}}_{\tau_{1}, T}[\mathbbm{1}_{\{\tau_{1}=\tau_{1}'\}}\psi(\tau_{1}', \tau_{2})]\\
&= \mathcal{E}^{g\mathbbm{1}_{\{\tau_{1}=\tau_{1}'\}}}_{\tau_{1}, \tau_{1}'\vee\tau_{2}}[\mathbbm{1}_{\{\tau_{1}=\tau_{1}'\}}\psi(\tau_{1}', \tau_{2})]\\ 
&= \mathbbm{1}_{\{\tau_{1}=\tau_{1}'\}}\mathcal{E}^{g}_{\tau_{1}, \tau_{1}'\vee\tau_{2}}[\psi(\tau_{1}', \tau_{2})].
\end{align*}
By using the same type of arguments, we can show that 
$$\mathbbm{1}_{\{\tau_{2}=\tau_{2}'\}}\rho_{\tau_{2}, \tau_{1}\vee\tau_{2}}[\psi(\tau_{1}, \tau_{2})] = \mathbbm{1}_{\{\tau_{2}=\tau_{2}'\}}\rho_{\tau_{2}', \tau_{1}\vee\tau_{2}'}[\psi(\tau_{1}, \tau_{2}')].$$
We now show (vi ter). Let $A$ be $\mathcal{F}_{\tau_{1}\wedge\tau_{2}}$-measurable, and $\tau_{1}\wedge\tau_{2} = \tau_{1}$ a.s. on $A$, and $\tau_{1}\vee\tau_{2} = \tau_{2}$ a.s. on $A$. We have:
\begin{align*}
\mathbbm{1}_{A}\mathcal{E}^{g}_{\tau_{1}\wedge\tau_{2}, \tau_{1}\vee\tau_{2}}[\psi(\tau_{1}\wedge\tau_{2}, \tau_{1}\vee\tau_{2})] &= \mathcal{E}^{g\mathbbm{1}_{A}}_{\tau_{1}\wedge\tau_{2}, \tau_{1}\vee\tau_{2}}[\mathbbm{1}_{A}\psi(\tau_{1}\wedge\tau_{2}, \tau_{1}\vee\tau_{2})]\\
&= \mathcal{E}^{g\mathbbm{1}_{A}}_{\tau_{1}\wedge\tau_{2}, \tau_{1}\vee\tau_{2}}[\mathbbm{1}_{A}\psi(\tau_{1}, \tau_{2})]\\
&= \mathbbm{1}_{A}\mathcal{E}^{g}_{\tau_{1}\wedge\tau_{2}, \tau_{1}\vee\tau_{2}}[\psi(\tau_{1}, \tau_{2})].
\end{align*}
Moreover,
\begin{align*}
\mathbbm{1}_{A}\mathcal{E}^{g}_{\tau_{1}\wedge\tau_{2}, \tau_{1}\vee\tau_{2}}[\psi(\tau_{1}\wedge\tau_{2}, \tau_{1}\vee\tau_{2})] &= \mathcal{E}^{g\mathbbm{1}_{A}}_{\tau_{1}\wedge\tau_{2}, \tau_{1}\vee\tau_{2}}[\mathbbm{1}_{A}\psi(\tau_{1}\wedge\tau_{2}, \tau_{1}\vee\tau_{2})]\\
&= \mathcal{E}^{g^{\tau_{1}\vee\tau_{2}}\mathbbm{1}_{A}}_{\tau_{1}\wedge\tau_{2}, T}[\mathbbm{1}_{A}\psi(\tau_{1}, \tau_{2})]\\
&= \mathcal{E}^{g^{\tau_{2}}\mathbbm{1}_{A}}_{\tau_{1}\wedge\tau_{2}, T}[\mathbbm{1}_{A}\psi(\tau_{1}, \tau_{2})]\\
&= \mathcal{E}^{g\mathbbm{1}_{A}}_{\tau_{1}\wedge\tau_{2}, \tau_{2}}[\mathbbm{1}_{A}\psi(\tau_{1}, \tau_{2})]\\
&= \mathbbm{1}_{A}\mathcal{E}^{g}_{\tau_{1}\wedge\tau_{2}, \tau_{2}}[\psi(\tau_{1}, \tau_{2})] = \mathbbm{1}_{A}\mathcal{E}^{g}_{\tau_{1}, \tau_{2}}[\psi(\tau_{1}, \tau_{2})].
\end{align*}
By the same arguments, we can show that 
$$\mathbbm{1}_{A}\mathcal{E}^{g}_{\tau_{1}\wedge\tau_{2}, \tau_{1}\vee\tau_{2}}[\psi(\tau_{1}\vee\tau_{2}, \tau_{1}\wedge\tau_{2})] = \mathbbm{1}_{A}\mathcal{E}^{g}_{\tau_{1}, \tau_{2}}[\psi(\tau_{2}, \tau_{1})],$$
which shows (vi ter).\\
[0.2cm] 
The same type of reasoning can be used to prove Property \ref{TEXPROperty-3-1} on $\rho$ used in Section \ref{Section4}.

\subsection{Dynamic concave utilities}
In this example, $p=+\infty$. A representation result with an explicit form for the penalty term, for dynamic concave utilities (DCUs) was established in \cite{Delbaen}. By the results of \cite{Delbaen}, a dynamic concave utility $u_{S, \tau}: L^{\infty}(\mathcal{F}_{\tau}) \to L^{\infty}(\mathcal{F}_{S})$ satisfies the following representation:
\begin{equation} \label{TEXequation-3-5-2}
\begin{aligned}
u_{S, \tau}(\eta) &=\essinf_{Q: Q\sim P,Q=P \text{ on } \cf_S  } \Big\{E_{Q}[\eta|\mathcal{F}_{S}]+c_{S,\tau}(Q)  \Big\}\\
&= \essinf_{Q \in \mathcal{Q}_{S}} E_{Q}[\eta + \int^{\tau}_{S}f(u, \psi_{u}^{Q})du|\mathcal{F}_{S}],
\end{aligned}
\end{equation}
where the function $f$ is  such that $f(\cdot,\cdot,x)$ is predictable for any $x$; $f$ is a proper, convex function in the space variable $x$, and valued in $[0,+\infty]$, and the process $(\psi_{t}^{Q})$ is the process from the Doleans-Dade exponential representation for the density process $(Z^{Q}_{t})$, where
$Z_{t}^{Q} = \frac{dQ}{dP}|_{\mathcal{F}_{t}},$
and 
$$\mathcal{Q}_{S} = \{Q: Q \sim P, \psi^{Q}_{t} = 0, dt \otimes dP \text{ a.e. on } [[0, S[[, \; E_{Q}[\int^{T}_{S}f(s, \psi^{Q}_{s})ds] < +\infty \}.$$
As noted also in \cite{Grigorova-Quenez-Peng-1}, the dynamic concave utility $u_{S, \tau}(\cdot)$ depends on the second index only via the penalty term $c_{S, \tau}$. The DCUs satisfy the properties: (i) (with $p = +\infty$), (ii) (admissibility), (iii) (knowledge preserving property), (iv) (monotonicity), (v) (consistency). Moreover, $u_{S, \tau}(\cdot)$ satisfies (vi) (``generalized zero-one law''). Indeed, from the representation result \eqref{TEXequation-3-5-2}, we get: for $A \in \mathcal{F}_{S}$, and for $\tau, \tau'$, such that $\tau = \tau'$ on $A$,
\begin{align*}
\mathbbm{1}_{A}u_{S, \tau}(\eta) &= \essinf_{Q \in \mathcal{Q}_{S}}E_{Q}[\mathbbm{1}_{A}\eta + \mathbbm{1}_{A}\int^{\tau}_{S}f(u, \psi^{Q}_{u})du|\mathcal{F}_{S}]\\
&= \mathbbm{1}_{A} \essinf_{Q \in \mathcal{Q}_{S}}E_{Q}[\eta + \int^{\tau'}_{S}f(u, \psi^{Q}_{u})du|\mathcal{F}_{S}] = \mathbbm{1}_{A}u_{S, \tau'}(\eta).
\end{align*}
Furthermore, $u_{S, \tau}[\cdot]$ satisfies (viii) (left-upper-semicontinuity (LUSC) along Bermudan stopping times with respect to the terminal condition and the terminal time) (cf. Remark 3.11 in \cite{Grigorova-Quenez-Peng-1}), and property (ix) (cf. Remark 3.12 in \cite{Grigorova-Quenez-Peng-1}). We assume moreover that $u_{S, \tau}$ satisfies property (vii) (Fatou property with respect to terminal condition), that is, we deal with the DCUs of the form \eqref{TEXequation-3-5-2} satisfying Fatou property.\\
[0.2cm]
Let us check that $u_{S, \tau}$ satisfies the property (vi bis).\\
[0.2cm] 
For this, we first recall that, if $\eta$ is integrable, then: $\mathbbm{1}_{\{\tau=\tau'\}}E[\eta|\mathcal{F}_{\tau}] = \mathbbm{1}_{\{\tau=\tau'\}}E[\eta|\mathcal{F}_{\tau'}]$.\\
[0.2cm]
Moreover, if $Q \sim P$, then we will identify it with: $(\frac{dQ}{dP}|_{\mathcal{F}_{t}} = Z_{t}^{Q}, \text{ for } t \in [0, T])$. We also identify $\frac{dQ}{dP}|_{\mathcal{F}_{t}} = Z_{t}^{Q}$ with the process $(\psi_{s}^{Q})$ from the Doleans-Dade exponential representation of the density process.\\
[0.2cm]
By abuse of notation, we will write: $(\rho_{s}^{Q}) \in \mathcal{Q}_{\tau}$, if $Q \in \mathcal{Q}_{\tau}$.\\ 
[0.2cm]
Let $Q \in \mathcal{Q}_{\tau}$, then: on $\{\tau = \tau'\}$, the process $(\rho_{s}^{Q})$ satisfies $\rho_{s}^{Q} = 0$ $dt \otimes dP$-a.e. on $[[0, \tau'[[$, and $E_{Q}[\int^{T}_{\tau'}f(s, \rho_{s}^{Q})ds] < +\infty$. We thus have:
\begin{align*}
&\mathbbm{1}_{\{\tau_{1} = \tau_{1}'\}}u_{\tau_{1}, \tau_{1} \vee \tau_{2}}(\psi(\tau_{1}, \tau_{2}))\\ 
&= \mathbbm{1}_{\{\tau_{1} = \tau_{1}'\}}\essinf_{Q \in \mathcal{Q}_{\tau_{1}}}E_{Q}[\psi(\tau_{1}, \tau_{2}) + \int^{\tau_{1} \vee \tau_{2}}_{\tau_{1}}f(u, \psi^{Q}_{u})du|\mathcal{F}_{\tau_{1}}]\\
&= \mathbbm{1}_{\{\tau_{1} = \tau_{1}'\}}\essinf_{Q \in \mathcal{Q}_{\tau_{1}'}}E_{Q}[\mathbbm{1}_{\{\tau_{1} = \tau_{1}'\}}\psi(\tau_{1}, \tau_{2}) + \mathbbm{1}_{\{\tau_{1} = \tau_{1}'\}}\int^{\tau_{1} \vee \tau_{2}}_{\tau_{1}}f(u, \psi^{Q}_{u})du|\mathcal{F}_{\tau_{1}}],
\end{align*}
where we have used that $\mathbbm{1}_{\{\tau_{1} = \tau_{1}'\}}$ is $\mathcal{F}_{\tau_{1}}$-measurable.\\
[0.2cm]
Hence, by using that $\psi$ is bi-admissible,
\begin{align*}
&\mathbbm{1}_{\{\tau_{1} = \tau_{1}'\}}u_{\tau_{1}, \tau_{1} \vee \tau_{2}}(\psi(\tau_{1}, \tau_{2}))\\
&= \mathbbm{1}_{\{\tau_{1} = \tau_{1}'\}}\essinf_{Q \in \mathcal{Q}_{\tau_{1}'}}E_{Q}[\mathbbm{1}_{\{\tau_{1} = \tau_{1}'\}}\psi(\tau_{1}', \tau_{2}) + \mathbbm{1}_{\{\tau_{1} = \tau_{1}'\}}\int^{\tau_{1}' \vee \tau_{2}}_{\tau_{1}'}f(u, \psi^{Q}_{u})du|\mathcal{F}_{\tau_{1}}]\\
&= \mathbbm{1}_{\{\tau_{1} = \tau_{1}'\}}\essinf_{Q \in \mathcal{Q}_{\tau_{1}'}}\mathbbm{1}_{\{\tau_{1} = \tau_{1}'\}}E_{Q}[\psi(\tau_{1}', \tau_{2}) + \int^{\tau_{1}' \vee \tau_{2}}_{\tau_{1}'}f(u, \psi^{Q}_{u})du|\mathcal{F}_{\tau_{1}}]\\
&= \mathbbm{1}_{\{\tau_{1} = \tau_{1}'\}}\essinf_{Q \in \mathcal{Q}_{\tau_{1}'}}\mathbbm{1}_{\{\tau_{1} = \tau_{1}'\}}E_{Q}[\psi(\tau_{1}', \tau_{2}) + \int^{\tau_{1}' \vee \tau_{2}}_{\tau_{1}'}f(u, \psi^{Q}_{u})du|\mathcal{F}_{\tau_{1}'}]\\
&= \mathbbm{1}_{\{\tau_{1} = \tau_{1}'\}}u_{\tau_{1}', \tau_{1}' \vee \tau_{2}}(\psi(\tau_{1}', \tau_{2})).
\end{align*}
Hence, (vi bis) holds true.\\
[0.2cm]
Let us now check (vi ter).\\
[0.2cm]
Let $A$ be $\mathcal{F}_{\tau_{1} \wedge \tau_{2}}$-measurable and let $\tau_{1}, \tau_{2}$ be such that $\tau_{1} \wedge \tau_{2} = \tau_{1}$ a.s. on $A$ and $\tau_{1} \vee \tau_{2} = \tau_{2}$ a.s. on $A$. Then, by using that $A$ is $\mathcal{F}_{\tau_{1} \wedge \tau_{2}}$-measurable,
\begin{align*}
&\mathbbm{1}_{A}u_{\tau_{1} \wedge \tau_{2}, \tau_{1} \vee \tau_{2}}(\psi(\tau_{1} \wedge \tau_{2}, \tau_{1} \vee \tau_{2}))\\
&= \mathbbm{1}_{A}\essinf_{Q \in \mathcal{Q}_{\tau_{1} \wedge \tau_{2}}}E_{Q}[\psi(\tau_{1} \wedge \tau_{2}, \tau_{1} \vee \tau_{2}) + \int^{\tau_{1} \vee \tau_{2}}_{\tau_{1} \wedge \tau_{2}}f(u, \psi^{Q}_{u})du|\mathcal{F}_{\tau_{1} \wedge \tau_{2}}]\\
&= \mathbbm{1}_{A}\essinf_{Q \in \mathcal{Q}_{\tau_{1}}}E_{Q}[\mathbbm{1}_{A}\psi(\tau_{1} \wedge \tau_{2}, \tau_{1} \vee \tau_{2}) + \mathbbm{1}_{A}\int^{\tau_{1} \vee \tau_{2}}_{\tau_{1} \wedge \tau_{2}}f(u, \psi^{Q}_{u})du|\mathcal{F}_{\tau_{1} \wedge \tau_{2}}]\\
&= \mathbbm{1}_{A}\essinf_{Q \in \mathcal{Q}_{\tau_{1}}}E_{Q}[\mathbbm{1}_{A}\psi(\tau_{1}, \tau_{2}) + \mathbbm{1}_{A}\int^{\tau_{2}}_{\tau_{1}}f(u, \psi^{Q}_{u})du|\mathcal{F}_{\tau_{1} \wedge \tau_{2}}]\\
&= \mathbbm{1}_{A}\essinf_{Q \in \mathcal{Q}_{\tau_{1}}}\mathbbm{1}_{A}E_{Q}[\psi(\tau_{1}, \tau_{2}) + \int^{\tau_{2}}_{\tau_{1}}f(u, \psi^{Q}_{u})du|\mathcal{F}_{\tau_{1}}]\\
&= \mathbbm{1}_{A}u_{\tau_{1}, \tau_{2}}(\psi(\tau_{1}, \tau_{2})).
\end{align*}
Also,
\begin{align*}
&\mathbbm{1}_{A}u_{\tau_{1} \wedge \tau_{2}, \tau_{1} \vee \tau_{2}}(\psi(\tau_{1} \wedge \tau_{2}, \tau_{1} \vee \tau_{2}))\\
&= \mathbbm{1}_{A}\essinf_{Q \in \mathcal{Q}_{\tau_{1} \wedge \tau_{2}}}E_{Q}[\mathbbm{1}_{A}\psi(\tau_{1} \wedge \tau_{2}, \tau_{1} \vee \tau_{2}) + \mathbbm{1}_{A}\int^{\tau_{1} \vee \tau_{2}}_{\tau_{1} \wedge \tau_{2}}f(u, \psi^{Q}_{u})du|\mathcal{F}_{\tau_{1} \wedge \tau_{2}}]\\
&= \mathbbm{1}_{A}\essinf_{Q \in \mathcal{Q}_{\tau_{1} \wedge \tau_{2}}}E_{Q}[\mathbbm{1}_{A}\psi(\tau_{1}, \tau_{2}) + \mathbbm{1}_{A}\int^{\tau_{1} \vee \tau_{2}}_{\tau_{1} \wedge \tau_{2}}f(u, \psi^{Q}_{u})du|\mathcal{F}_{\tau_{1} \wedge \tau_{2}}]\\
&= \mathbbm{1}_{A}u_{\tau_{1} \wedge \tau_{2}, \tau_{1} \vee \tau_{2}}(\psi(\tau_{1}, \tau_{2})).
\end{align*}

\newpage 
\section{Appendix}
\begin{proof}[Proof of Proposition \ref{TEXproposition5555}]
\textbf{{i)}} This is a consequence of property $(i)$ of $\rho$ and a well-known property of the essential supremum.\\
[0.2cm]
\textbf{{ii)}} It is sufficient to show that the family $(\rho_{S, \tau_{1} \vee ... \vee \tau_{d}}[\psi(\tau_{1}, ..., \tau_{d})])_{(\tau_{1}, ..., \tau_{d}) \in \Theta_{S}^{d}}$ is directed upwards.\\
Let $(\tau_{1}, ..., \tau_{d}) \in \Theta_{S}^{d}$ and $(\tau_{1}', ..., \tau_{d}') \in \Theta_{S}^{d}$. We define the set: $A \coloneqq$\\ 
$\{\rho_{S, \tau_{1}' \vee ... \vee \tau_{d}'}[\psi(\tau_{1}', ..., \tau_{d}')] \leq \rho_{S, \tau_{1} \vee ... \vee \tau_{d}}[\psi(\tau_{1}, ..., \tau_{d})] \}$. Trivially, $A \in \mathcal{F}_{S}$.\\
[0.2cm]
We set: for $i \in \{1, ..., d\}$, $\nu_{i} \coloneqq \tau_{i}\mathbbm{1}_{A} + \tau_{i}'\mathbbm{1}_{A^{c}}$. We have, by Property \ref{TEXPROperty-3-1},
\begin{align*}
\rho_{S, \nu_{1} \vee ... \vee \nu_{d}}[\psi(\nu_{1}, ..., \nu_{d})] &= \rho_{S, \nu_{1} \vee ... \vee \nu_{d}}[\psi(\nu_{1}, ..., \nu_{d})]\mathbbm{1}_{A} + \rho_{S, \nu_{1} \vee ... \vee \nu_{d}}[\psi(\nu_{1}, ..., \nu_{d})]\mathbbm{1}_{A^{c}}\\
&= \rho_{S, \tau_{1} \vee ... \vee \tau_{d}}[\psi(\tau_{1}, ..., \tau_{d})]\mathbbm{1}_{A} + \rho_{S, \tau_{1}' \vee ... \vee \tau_{d}'}[\psi(\tau_{1}', ..., \tau_{d}')]\mathbbm{1}_{A^{c}}\\
&= \max(\rho_{S, \tau_{1} \vee ... \vee \tau_{d}}[\psi(\tau_{1}, ..., \tau_{d})], \rho_{S, \tau_{1}' \vee ... \vee \tau_{d}'}[\psi(\tau_{1}', ..., \tau_{d}')]).
\end{align*}
This shows that the family $(\rho_{S, \tau_{1} \vee ... \vee \tau_{d}}[\psi(\tau_{1}, ..., \tau_{d})])_{(\tau_{1}, ..., \tau_{d}) \in \Theta_{S}^{d}}$ is stable by maximum (hence, it is directed upwards). Hence, by a well-known property of the essential supremum, statement (ii) holds true.\\
[0.2cm]
\textbf{{iii)}} By statement (i), the family $(V(S))_{S \in \Theta}$ is admissible. Moreover, $(V(S))_{S \in \Theta}$ is $p$-integrable by Assumption \ref{TEXAssumption3-3}.\\
[0.2cm]
Let $S, S'$ in $\Theta$ be such that $S \leq S'$ a.s. We will show that $V(S) \geq \rho_{S, S'}[V(S')]$ a.s.\\
[0.2cm]
By statement (ii), there exists a sequence $(\tau^{n}_{1}, \tau^{n}_{2}, ..., \tau^{n}_{d}) \in \Theta^{n}_{S'}$ such that $V(S') = \lim_{n \to +\infty}\uparrow\rho_{S', \tau^{n}_{1} \vee \tau^{n}_{2} \vee ... \vee \tau^{n}_{d}}[\psi(\tau^{n}_{1}, \tau^{n}_{2}, ..., \tau^{n}_{d})]$.\\
[0.2cm]
Hence, 
\begin{align*}
\rho_{S, S'}[V(S')] &= \rho_{S, S'}[\lim_{n \to +\infty}\rho_{S', \tau^{n}_{1} \vee \tau^{n}_{2} \vee ... \vee \tau^{n}_{d}}[\psi(\tau^{n}_{1}, \tau^{n}_{2}, ..., \tau^{n}_{d})]]\\
&\leq \liminf_{n \to +\infty}\rho_{S, S'}[\rho_{S', \tau^{n}_{1} \vee \tau^{n}_{2} \vee ... \vee \tau^{n}_{d}}[\psi(\tau^{n}_{1}, \tau^{n}_{2}, ..., \tau^{n}_{d})]],
\end{align*}
where we have used the monotone Fatou property with respect to terminal condition (property (vii) of $\rho$).\\
[0.2cm]
By the consistency property of $\rho$, 
\begin{align*}
\rho_{S, S'}[\rho_{S', \tau^{n}_{1} \vee \tau^{n}_{2} \vee ... \vee \tau^{n}_{d}}[\psi(\tau^{n}_{1}, \tau^{n}_{2}, ..., \tau^{n}_{d})]] = \rho_{S, \tau^{n}_{1} \vee \tau^{n}_{2} \vee ... \vee \tau^{n}_{d}}[\psi(\tau^{n}_{1}, \tau^{n}_{2}, ..., \tau^{n}_{d})].
\end{align*}
Finally, 
\begin{align*}
\rho_{S, S'}[V(S')] \leq  \liminf_{n \to +\infty}\rho_{S, \tau^{n}_{1} \vee \tau^{n}_{2} \vee ... \vee \tau^{n}_{d}}[\psi(\tau^{n}_{1}, \tau^{n}_{2}, ..., \tau^{n}_{d})] \leq V(S).
\end{align*}
This shows the $(\Theta, \rho)$-supermartingale property.

\end{proof}

\newpage
\bibliographystyle{plainnat}

\end{document}